\documentclass[11pt,letterpaper,oneside]{article}
\usepackage[letterpaper,margin=1in]{geometry}
\usepackage[T1]{fontenc}
\usepackage{lmodern,microtype,amsmath,amssymb,amsthm,mathtools}
\microtypesetup{protrusion=false}
\usepackage{booktabs,array,tabularx,enumitem,float,needspace,graphicx}
\floatstyle{ruled}
\newfloat{algorithm}{tbp}{loa}
\floatname{algorithm}{Algorithm}
\usepackage{algpseudocode}
\usepackage[dvipsnames]{xcolor}
\usepackage{xurl}
\usepackage[colorlinks=true,linkcolor=MidnightBlue,citecolor=MidnightBlue,urlcolor=MidnightBlue]{hyperref}
\usepackage{bookmark}

\hypersetup{pdftitle={Adaptive Variance Reduction across Smoothness and Noise Conditions}}
\renewenvironment{abstract}{\begin{center}\bfseries Abstract\end{center}\begin{quote}\normalfont\normalsize}{\end{quote}}
\newcommand{\E}{\mathbb E}
\newcommand{\R}{\mathbb R}
\newcommand{\norm}[1]{\left\|#1\right\|}
\newcommand{\inner}[2]{\left\langle#1,#2\right\rangle}
\newcommand{\N}{\mathcal N}
\newcommand{\Deltaf}{\Delta_f}
\newcommand{\xinit}{x_{\mathrm{init}}}
\newtheorem{theorem}{Theorem}
\newtheorem{lemma}{Lemma}
\newtheorem{proposition}{Proposition}
\newtheorem{corollary}{Corollary}
\theoremstyle{definition}
\newtheorem{assumption}{Assumption}
\theoremstyle{remark}

\setlist{itemsep=3pt,topsep=5pt}
\allowdisplaybreaks[1]
\title{Adaptive Variance Reduction\\across Smoothness and Noise Conditions}
\author{
{Wei Jiang}$^{1,2}$ \quad
{Yibo Wang}$^{2,3}$ \quad
{Sifan Yang}$^{2,3}$ \quad
{Lijun Zhang}$^{2,3}$ \quad
{Zechao Li}$^{1}$ \\[0.5em]
\normalsize$^1$School of Computer Science and Engineering,
Nanjing University of Science and Technology\\
\normalsize$^2$State Key Laboratory of Novel Software Technology, Nanjing University\\
\normalsize$^3$School of Artificial Intelligence, Nanjing University\\
\normalsize \texttt{12025220@njust.edu.cn}, \quad \texttt{wangyb@lamda.nju.edu.cn}, \quad \texttt{yangsf@lamda.nju.edu.cn}, \\ \normalsize \texttt{zhanglj@lamda.nju.edu.cn}, \quad \texttt{zechao.li@njust.edu.cn}
}
\date{}
\begin{document}
\maketitle

\begin{abstract}
In this paper, we study adaptive stochastic non-convex optimization under unknown smoothness and noise conditions.
Our algorithms require no smoothness constant, noise bound, initial function gap, or oracle-regime label.
We first propose a fixed-budget method that combines normalized momentum, recursive variance reduction, and a fresh gradient stage, followed by fresh validation of the candidates.
We prove that the same output has expected gradient norm $O(T^{-1/4})$ under objective smoothness and bounded noise variance, $O(T^{-1/3})$ under average smoothness, and $O(T^{-1/2})$ for noiseless gradients, where $T$ counts individual gradient evaluations.

We then develop an anytime algorithm whose three optimization paths preserve their endpoints across reporting and budget-extension boundaries. The same three rates hold at every deterministic query budget, although the algorithm receives no information about the budget $T$. The analysis combines a scale potential, a uniform second-moment bound, and accounting for candidates retained through unfinished blocks.
An optional sample-mean selector further reduces cumulative validation work from $O(T)$ to $O(T^{2/3})$ while preserving these rates.
\end{abstract}

\section{Introduction}
\label{sec:introduction}
We investigate stochastic non-convex optimization in the form
\begin{equation}
 \min_{x\in\R^d} f(x),
 \label{eq:problem}
\end{equation}
where $f$ is smooth and bounded below, and the oracle provides unbiased stochastic gradients $g(x;\xi)$.
The goal is to find a point with small expected gradient norm using few gradient evaluations.
Under the bounded noise variance assumption, stochastic gradient methods achieve an $O(T^{-1/4})$ convergence rate \cite{ghadimi2013stochastic}, where $T$ counts individual gradient evaluations.
If the stochastic oracle also satisfies average smoothness, variance reduction improves the rate to $O(T^{-1/3})$ \cite{fang2018spider}.
For exact gradients, gradient descent attains an $O(T^{-1/2})$ convergence rate \cite{ghadimi2013stochastic}, matching the corresponding worst-case lower bound \cite{carmon2020stationaryI}.

On the other hand, adaptive variance reduction methods remove the need to know smoothness constants and noise bounds, but they are usually analyzed under a specified oracle condition \cite{levy2021stormplus,liu2022metastorm,jiang2024adastorm}.
For example, a guarantee based on average smoothness may not remain valid when only the objective is smooth.
This leads to the following question:
\emph{Can one algorithm attain the rate of each oracle setting without knowing which setting applies?}
We require these guarantees for the same returned point, rather than different outputs selected with knowledge of the assumptions.
\begin{table}[t]
\centering
\caption{Expected gradient-norm rates for a single output per method, with problem parameters held fixed.
The budget $T$ includes initialization and validation queries.}
\label{tab:rates}
\begin{tabular}{@{}lccc@{}}
\toprule
Oracle condition & Fixed budget & Restarting anytime & Continuing anytime\\
\midrule
Objective smoothness & $O(T^{-1/4})$ & $O(T^{-1/4})$ & $O(T^{-1/4})$\\
Average smoothness & $O(T^{-1/3})$ & $O(T^{-1/3})$ & $O(T^{-1/3})$\\
Noiseless gradients & $O(T^{-1/2})$ & $O(T^{-1/2})$ & $O(T^{-1/2})$\\
\midrule
Budget extension without restarts & N/A & No & Yes\\
Anytime method & No & Yes & Yes\\
\bottomrule
\end{tabular}
\end{table}
\subsection{Our contributions}

Our main contribution concerns adaptation to different smoothness and noise conditions while optimization paths continue from their current endpoints.
We summarize our results in Table~\ref{tab:rates}, for the fixed-budget construction, its geometric restarting wrapper, and the continuing anytime method.
None of these methods receives problem constants or a label specifying the oracle condition.

\paragraph{One output with three guarantees.}
For a prescribed budget, Theorem~\ref{thm:simultaneous} combines normalized momentum~(S), recursive variance reduction~(M), and a fresh gradient stage~(D) with fresh validation.
Running a constant number of specialists and selecting a candidate gives simultaneous guarantees under three different smoothness and noise conditions.
The proposed algorithm allows the three stages to run consecutively along one trajectory, each starting from the preceding endpoint.
We further develop a geometric restarting technique, which gives a baseline for the anytime requirement.

\paragraph{Analysis for continuing paths.}
For an unknown budget, Theorem~\ref{thm:anytime} gives the same rates under objective smoothness, average smoothness, and exact gradients.
On paths S and M, each block starts at the preceding block's endpoint.
All three paths retain their current iterates across reporting checkpoints.
Reporting reads a stored candidate while optimization continues. This endpoint requirement separates the construction from the restarting wrapper.
Its main ingredients are a scale potential that bounds cumulative estimation error, a uniform second-moment estimate for the adaptive scale, and deterministic accounting for the reuse of candidates through unfinished blocks.
A scalar accumulated-gradient normalization further supplies the noiseless guarantee with one gradient query per update.
The D path also has a stochastic guarantee with logarithmic factors, and the S path supplies the log-free basic-regime rate (Appendix~\ref{sec:adagrad-comparison}).
Our continuing construction uses three parallel persistent paths, each retaining its own endpoint, and at most two queries per optimization update.
Linear validation gives a selection error of $O(\sigma/\sqrt T)$. Smaller validation batches preserve the three convergence exponents with cumulative validation work $O(T^{2/3})$ and additive selection error $O(\sigma T^{-1/3})$.

\subsection{Related work}

\paragraph{Stochastic gradient methods and variance reduction.}
Stochastic gradient methods attain $O(T^{-1/4})$ expected gradient norm for smooth objectives with bounded noise variance \cite{ghadimi2013stochastic}.
Normalized momentum attains the same rate with constant batch size and without a bound on stochastic gradients \cite{cutkosky2020momentum}.
Variance reduction was developed for finite-sum problems through methods such as SVRG \cite{johnson2013svrg} and SARAH \cite{nguyen2017sarah}.
SVRG uses a periodically computed reference gradient, while SARAH recursively corrects an estimate using gradients at consecutive points.
Non-convex SVRG and SCSG extend variance reduction beyond convex objectives \cite{reddi2016nonconvex,lei2017scsg}.
For example, non-convex SVRG obtains $O(m+m^{2/3}\epsilon^{-2})$ gradient complexity for a sum of $m$ smooth functions and expected gradient norm at most $\epsilon$ \cite{reddi2016nonconvex}.
Under average smoothness, SPIDER improves the stochastic rate to $O(T^{-1/3})$ through recursive differences \cite{fang2018spider}.
SpiderBoost permits a constant step size and extends this approach to composite objectives \cite{wang2019spiderboost}.
PAGE uses random switching between a fresh estimate and a recursive correction, attaining $O(B+\sqrt B\,\epsilon^{-2})$ stochastic gradient complexity with $B$ of order $\min\{m,\max\{1,\sigma^2/\epsilon^2\}\}$ in its finite-sum or online setting \cite{li2021page}.
These constructions use periodic or randomly triggered batches.
STORM instead uses recursive momentum with constant query cost and proves $\widetilde O(T^{-1/3})$ under samplewise smoothness and bounded stochastic gradients \cite{cutkosky2019storm}.

\paragraph{Adaptive step sizes and recursive momentum.}
AdaGrad introduced step sizes based on accumulated gradient information for online and stochastic convex optimization \cite{duchi2011adaptive}.
Later analyses established non-convex convergence and adaptation to gradient noise for AdaGrad-type methods \cite{ward2019adagrad,li2019adaptive}.
Subsequent scalar AdaGrad-Norm analyses remove bounded-gradient assumptions under affine variance \cite{faw2022power,wang2023adagrad}. Under bounded variance, the analysis in \cite{wang2023adagrad} yields a $\widetilde O(T^{-1/4})$ expected gradient-norm bound for a uniformly sampled iterate. Parameter-free stochastic optimization through hyperparameter search is also studied in \cite{attia2024free}.
For finite sums, AdaSpider obtains $\widetilde O(m+\sqrt m\,\epsilon^{-2})$ complexity without knowing the smoothness constant or target accuracy \cite{kavis2022adaspider}.
STORM+ requires no smoothness or noise parameters and proves $O(T^{-1/3})$ expected gradient norm \cite{levy2021stormplus}.
Its analysis assumes samplewise smoothness, bounded stochastic gradients, and a bounded objective.
META-STORM removes the bounded-objective condition and provides adaptive guarantees under average smoothness \cite{liu2022metastorm}. Its original variant assumes bounded differences across samples between stochastic gradients at the same point, while its SG variant assumes bounded stochastic gradients.
Ada-STORM achieves $O(T^{-1/3})$ expected gradient norm under average smoothness and bounded variance, without bounded gradients or objective range \cite{jiang2024adastorm}.
However, its prescribed-horizon construction uses an initial batch of order $T^{1/3}$, and its doubling conversion restarts each run at the initial point.

\section{Problem setting}
\label{sec:setup}

We study problem~\eqref{eq:problem} with a fixed initial point $\xinit\in\R^d$.
The symbol $\norm{\cdot}$ denotes the Euclidean norm, $f_*:=\inf_{x\in\R^d}f(x)$, and $g(x;\xi)$ denotes a stochastic gradient.
The following two assumptions apply throughout.

\begin{assumption}[Objective smoothness and initial gap]
\label{ass:objective}
The function $f$ is differentiable, $f_*>-\infty$, and, for all $x,y\in\R^d$,
\[
 \norm{\nabla f(x)-\nabla f(y)}\le L\norm{x-y}.
\]
The initial point satisfies $f(\xinit)-f_*\le\Deltaf$, where $L>0$ and $\Deltaf\ge0$.
\end{assumption}

\begin{assumption}[Unbiased gradients and bounded variance]
\label{ass:variance}
For every $x\in\R^d$,
\[
 \E g(x;\xi)=\nabla f(x),\qquad
 \E\norm{g(x;\xi)-\nabla f(x)}^2\le\sigma^2,
\]
where $\sigma\ge0$.
\end{assumption}

Objective smoothness only controls the mean gradient.
Our first guarantee uses exactly these two assumptions.
The second guarantee uses the following additional condition.

\begin{assumption}[Average smoothness]
\label{ass:average}
For all $x,y\in\R^d$, the gradients evaluated with the same random sample satisfy
\begin{equation}
 \E\norm{g(x;\xi)-g(y;\xi)}^2
 \le \bar L^2\norm{x-y}^2
 \label{eq:average-smoothness}
\end{equation}
for some $\bar L\ge0$.
\end{assumption}

The constants $L$ and $\bar L$ refer to different inequalities.
By Jensen's inequality, Assumption~\ref{ass:average} implies objective smoothness with constant $\bar L$, but objective smoothness alone does not imply Assumption~\ref{ass:average}.
We keep the two symbols to distinguish the conditions used in the proof.

The third guarantee concerns a noiseless oracle, which returns $g(x;\xi)=\nabla f(x)$ almost surely and permits $\sigma=0$.
Such an oracle satisfies average smoothness with $\bar L=L$.
The algorithm is not told whether gradients are exact.

\paragraph{Oracle access and evaluation cost.}
Before a fresh independent sample is drawn, the algorithm chooses one or two points using its past observations.
Fresh samples are independent of the preceding history, so the stated mean and variance bounds apply conditionally at adaptively chosen points.
The algorithm does not observe the noise separately or query function values.
The budget includes initialization, shared-sample evaluations, and candidate validation.

\paragraph{Performance and parameter knowledge.}
We bound $\E\norm{\nabla f(\widehat x)}$ for the returned point $\widehat x$.
The fixed-budget algorithm receives $\xinit$ and an integer budget $T$.
The anytime algorithm receives $\xinit$ without a total budget.
Neither receives $L$, $\bar L$, $\sigma$, or $\Deltaf$, and it does not receive the applicable noise condition. The anytime guarantees hold for each deterministic total-query budget, including interruptions during an update or validation batch.
We assume neither bounded gradients nor a bounded domain or objective range.

\paragraph{Notation for fixed-budget stages.}
In Sections~\ref{sec:algorithm}--\ref{sec:results}, we use $n$ for the number of updates in each of three stages.
Within a stage, $x_0$ denotes its starting point and $x_n$ its final iterate.
Here $\xinit$ always denotes the initial point of the whole algorithm.
The vector $v_t$ is the gradient estimator, and $\eta$ is the length of a nonzero normalized update.
These local indices restart in each stage, while the optimization point continues from the preceding final iterate.
Section~\ref{sec:anytime} introduces separate block and checkpoint indices for the three continuing paths.

\section{The fixed-budget algorithm}
\label{sec:algorithm}

\begin{algorithm}[t]
\caption{Parameter-free optimization with a prescribed budget}
\label{alg:three-stage}
\begin{algorithmic}[1]
\Require Initial point $\xinit$, integer budget $T\ge7$
\State $n\gets\lfloor T/7\rfloor$, $Y\gets\xinit$
\State $(\eta_S,\eta_M,\eta_D)\gets(n^{-3/4},n^{-2/3},n^{-1/2})$
\State $\gamma\gets n^{-1/2}$, $\beta\gets n^{-2/3}$
\For{$i=S,M,D$ \textbf{in this order}}
  \State $x_0\gets Y$
  \For{$t=0,\ldots,n-1$}
    \State Draw a fresh sample $\xi_t$
    \State Set $v_0\gets g(x_0;\xi_0)$ if $t=0$.
    \State For $t\ge1$, compute $v_t$ by \eqref{eq:S-recursion}, \eqref{eq:M-recursion}, or \eqref{eq:D-direction} for stage $i$.
    \State With probability $1/(t+1)$, set $X_i\gets x_t$
    \State $x_{t+1}\gets x_t-\eta_i\N(v_t)$
  \EndFor
  \State $Y\gets x_n$
\EndFor
\For{$i\in\{S,M,D\}$}
  \State $\widehat g_i\gets n^{-1}\sum_{j=1}^n g(X_i;\xi_{i,j}^{\mathrm{val}})$ using fresh samples
\EndFor
\State \Return $X_{\widehat i}$, where $\widehat i\in\arg\min_{i\in\{S,M,D\}}\norm{\widehat g_i}$
\end{algorithmic}
\end{algorithm}
Algorithm~\ref{alg:three-stage} runs three stages in order.
Stage S uses normalized momentum, stage M uses recursive variance reduction, and stage D uses a fresh gradient at every update.
Stage S starts at the initial point. Stages M and D each start at the preceding stage’s endpoint.
We run all three stages and select their output using fresh gradients.
For $n=\lfloor T/7\rfloor$, each stage makes $n$ normalized updates,
\begin{equation}
 x_{t+1}=x_t-\eta\N(v_t),\qquad
 \N(v)=\begin{cases}v/\norm v,&v\ne0,\\0,&v=0,\end{cases}
 \label{eq:normalized-step}
\end{equation}
and initializes its estimator with one fresh gradient $v_0=g(x_0;\xi_0)$.
The following recursions apply for $t=1,\ldots,n-1$.

\paragraph{Stage S.}
Set $\eta=n^{-3/4}$ and $\gamma=n^{-1/2}$.
The normalized momentum estimator is
\begin{equation}
 v_t=(1-\gamma)v_{t-1}+\gamma g(x_t;\xi_t).
 \label{eq:S-recursion}
\end{equation}
Objective smoothness controls the change of the mean gradient, and momentum reduces the noise.

\paragraph{Stage M.}
Set $\eta=\beta=n^{-2/3}$.
Using the same fresh sample at $x_t$ and $x_{t-1}$, the STORM estimator \cite{cutkosky2019storm} is
\begin{equation}
 v_t=g(x_t;\xi_t)
 +(1-\beta)\bigl[v_{t-1}-g(x_{t-1};\xi_t)\bigr].
 \label{eq:M-recursion}
\end{equation}
Average smoothness controls the error of the paired difference.

\paragraph{Stage D.}
Set $\eta=n^{-1/2}$ and use a fresh gradient directly:
\begin{equation}
 v_t=g(x_t;\xi_t),\qquad t=0,\ldots,n-1.
 \label{eq:D-direction}
\end{equation}
This stage uses one query per update. Its direction is exact under a noiseless oracle, so normalized descent gives the $n^{-1/2}$ guarantee.
The same stage is run when gradients are noisy, and validation protects the guarantees supplied by S and M in that case.

\paragraph{Output selection.}
Stage $i$ retains a uniform candidate $X_i$ from its iterates $\{x_0,\ldots,x_{n-1}\}$.
Reservoir sampling implements this choice by replacing $X_i$ by $x_t$ at iteration $t$ with probability $1/(t+1)$.
The endpoint $x_n$ initializes the next stage.
After all stages, form
\begin{equation}
 \widehat g_i=\frac1n\sum_{j=1}^n g(X_i;\xi_{i,j}^{\mathrm{val}}),
 \qquad i\in\{S,M,D\},
 \label{eq:validation}
\end{equation}
using fresh samples, and return the candidate minimizing $\norm{\widehat g_i}$.
Ties use the order S, M, D.

\begin{lemma}[Validation]
\label{lem:selection}
Let $X_S,X_M,X_D$ be any three candidates with finite expected gradient norm, chosen before validation.
Average $n$ fresh gradients at each candidate and return the one with the smallest norm of this average.
Under Assumption~\ref{ass:variance},
\begin{equation}
 \E\norm{\nabla f(\widehat x)}
 \le\min_{i\in\{S,M,D\}}\E\norm{\nabla f(X_i)}
          +\frac{\sqrt6\,\sigma}{\sqrt n}.
 \label{eq:proof-validation}
\end{equation}
\end{lemma}
The validation term is no larger than the target rates and vanishes for exact gradients. Thus selection requires neither an estimate of $\sigma$ nor knowledge of which stage has the strongest guarantee.

\paragraph{Budget allocation.}
S and D each use $n$ queries, M uses $2n-1$, and validation uses $3n$. The total is $7n-1$, including each stage's initialization. The random candidate retained by a stage is separate from its endpoint: the former is reserved for validation, while the latter initializes the next stage. This distinction lets the algorithm compare the stages without moving the optimization point back to a stored candidate.

\section{Convergence guarantees}
\label{sec:results}

The following theorem gives three guarantees for the same output of Algorithm~\ref{alg:three-stage}.
The algorithm uses the same parameters in all three cases and does not test which condition holds.

\begin{theorem}
\label{thm:simultaneous}
Suppose Assumptions~\ref{ass:objective} and~\ref{ass:variance} hold.
For any integer $T\ge7$, let $n=\lfloor T/7\rfloor$.
Algorithm~\ref{alg:three-stage} uses $7n-1\le T$ gradient evaluations and returns $\widehat x_T$ with the following guarantees.
\begin{equation}
 \E\norm{\nabla f(\widehat x_T)}\le
 \begin{cases}
 (\Deltaf+3L+7\sigma)n^{-1/4},&\text{under the two basic assumptions},\\
 (\Deltaf+3L+3\bar L+12\sigma)n^{-1/3},&\text{under average smoothness as well},\\
 (\Deltaf+4L)n^{-1/2},&\text{for noiseless gradients}.
 \end{cases}
 \label{eq:fixed-rates}
\end{equation}

The algorithm does not use $L$, $\bar L$, $\sigma$, $\Deltaf$, or the applicable oracle condition.
\end{theorem}
\paragraph{Rate comparison.} Since $n\ge T/14$, conversion to $T$ multiplies the displayed coefficients by at most $14^{1/2}<4$.
The first two exponents match the stochastic upper bounds in \cite{ghadimi2013stochastic,fang2018spider} and the respective worst-case lower bounds in \cite{arjevani2023lower}.
For exact gradients, the $T^{-1/2}$ upper bound follows from the zero-variance case of \cite{ghadimi2013stochastic}, with a matching lower bound in \cite{carmon2020stationaryI}.
Under the basic assumptions, stage S supplies the guarantee.
Average smoothness permits stage M to improve the rate.
For noiseless gradients, M has an exact estimator and controls its outgoing function gap. 
The fresh gradient stage D then supplies the faster guarantee.
Fresh validation preserves the best applicable guarantee with an error of order $\sigma/\sqrt n$.
Theorem~\ref{thm:simultaneous} obtains these three guarantees from one parameter-free algorithm.

\paragraph{Proof mechanism.}
The analysis bounds both the candidate gradient norm and the objective gap at each stage's endpoint. The outgoing-gap bound controls the random starting point of the next stage, while fresh validation preserves the best applicable candidate guarantee. Appendix~\ref{sec:proofs} gives the complete argument.

\section{The anytime algorithm}
\label{sec:anytime}
A doubling schedule can handle an unknown budget by running Algorithm~\ref{alg:three-stage} from $\xinit$ with budgets $7,14,28,\ldots$ and retaining the last completed output (Appendix~\ref{sec:restarting-baseline}).
To avoid such restarts, we obtain the same three rate exponents with paths that preserve their current endpoints across block and reporting boundaries.
Specifically, the algorithm maintains independent paths S, M, and D for the basic, average-smooth, and noiseless guarantees, respectively.
Each path reports a candidate, and fresh validation selects the output.

\subsection{Adaptive blocks for S and M}
\label{sec:anytime-SM}
Each path starts at $x=\xinit$ and maintains a block length $b=2$ and a scale $a=1$.
The scale controls the step length, and the block length controls how many samples contribute to the estimator.
At the beginning of a block, set $x_1=x$ and $v_1=g(x_1;\xi_1)$ with one fresh sample.
For $t=2,\ldots,b$, use
\begin{equation}
 \begin{aligned}
 \text{S:}\qquad v_t&=(1-1/t)v_{t-1}+g(x_t;\xi_t)/t,\\
 \text{M:}\qquad v_t&=g(x_t;\xi_t)+(1-1/t)\bigl(v_{t-1}-g(x_{t-1};\xi_t)\bigr).
 \end{aligned}
 \label{eq:block-estimators}
\end{equation}
S averages gradients observed within the block. M uses a shared-sample difference to correct for movement between successive points.
After computing $v_t$, take the normalized step
\begin{equation}
 x_{t+1}=x_t-\eta\N(v_t),\qquad
 \eta=\begin{cases}
 \sqrt{2}/(ab^{3/2}),&\text{for S},\\
 1/(ab),&\text{for M},
 \end{cases}
 \label{eq:block-parameters}
\end{equation}
where $\N$ is defined in \eqref{eq:normalized-step}.
The block uses $b$ queries for S and $2b-1$ for M.

\paragraph{Adapting the block parameters.}
Let $m$ count the estimates satisfying
\begin{equation}
 \norm{v_t}>4a\sqrt{2/b}.
 \label{eq:block-test}
\end{equation}
A block is successful when $m\ge b/2$.
Many large estimates support descent at the current accuracy, so success shortens the next block. Otherwise, the path lengthens it to average more samples.
In both cases, the next block starts from $x_{b+1}$, with parameters updated by Table~\ref{tab:block-updates}.
The scale update uses the completed block's length.
\begin{table}[ht]
\centering
\caption{Parameter updates after a block with $m$ large estimates.}
\label{tab:block-updates}
\begin{tabular}{@{}lcc@{}}
\toprule
 & S & M\\
\midrule
Scale & $a\leftarrow a+2m/b^2$ & $a\leftarrow a+\sqrt{2/b}\,m/b$\\[2pt]
Success, $m\ge b/2$ & $b\leftarrow\max\{b/2,2\}$ & $b\leftarrow\max\{b/4,2\}$\\[2pt]
Failure, $m<b/2$ & $b\leftarrow2b$ & $b\leftarrow4b$\\
\bottomrule
\end{tabular}
\end{table}
\paragraph{Reporting through unfinished blocks.}
A failed block has an index $t>b/2$ with $\norm{v_t}\le4a\sqrt{2/b}$.
When that block finishes, store the first such pre-update point as $Y$.
A successful block leaves $Y$ unchanged.
Using the second half ensures that the estimator has accumulated at least $b/2$ samples.
Every estimate produces an update, regardless of whether its norm passes the test.

At the start of a block, fix a reporting label equal to its starting point if $b=2$, or the stored $Y$ if $b>2$.
The latter exists because increasing $b$ first requires a failed block.
Assign this fixed label to every individual query of the block, including both queries of an M pair and the final query of the block.
Only after processing that last label may the new $Y$ replace the stored point.
At local query number $r$, replace the path's candidate by its current label with probability $1/r$, using an independent coin.
This reservoir is uniform over all query labels after any deterministic number of queries, including interruptions inside a block or pair.
Algorithm~\ref{alg:block-path} specifies the S/M worker, including its query counter and reservoir.
The scheduler executes one $\operatorname{Query}$ call and all updates enabled by its response, then reads $R$.
All state is retained between calls, including $\xi_t,u,x_t,x_{t-1}$ between the two responses of an M pair.

\begin{algorithm}[t]
\caption{A continuing S or M path}
\label{alg:block-path}
\begin{algorithmic}[1]
\Require Initial point $\xinit$ and path type S or M
\State $x\gets\xinit$, $a\gets1$, $b\gets2$, $q\gets0$, $R\gets\xinit$
\Statex \textbf{Query}$(z;\xi)$: evaluate $w\gets g(z;\xi)$ and increment $q\gets q+1$.
\Statex \hspace{\algorithmicindent} With probability $1/q$, set $R\gets\ell$ using an independent coin; return $w$.
\Loop
 \State Fix $\eta$ from \eqref{eq:block-parameters}; set $x_1\gets x$, $m\gets0$, and $y\gets\bot$.
 \State Set the fixed reporting label $\ell\gets x_1$ if $b=2$, and $\ell\gets Y$ otherwise.
 \For{$t=1,\ldots,b$}
  \State Fix $x_t$ (and $x_{t-1}$ for an M pair), then draw a fresh sample $\xi_t$.
  \State $u\gets\operatorname{Query}(x_t;\xi_t)$
  \If{path M and $t\ge2$}
   \State $w\gets\operatorname{Query}(x_{t-1};\xi_t)$; $v_t\gets u+(1-1/t)(v_{t-1}-w)$.
  \Else
   \State Set $v_t\gets u$ if $t=1$, and $v_t\gets(1-1/t)v_{t-1}+u/t$ otherwise.
  \EndIf
  \If{$\norm{v_t}>4a\sqrt{2/b}$}
   \State $m\gets m+1$
  \ElsIf{$t>b/2$ and $y=\bot$}
   \State $y\gets x_t$
  \EndIf
  \State $x_{t+1}\gets x_t-\eta\N(v_t)$
 \EndFor
 \If{$m<b/2$}
  \State $Y\gets y$
 \EndIf
 \State $x\gets x_{b+1}$
 \State Update $a$, then $b$, using Table~\ref{tab:block-updates}.
\EndLoop
\end{algorithmic}
\end{algorithm}

\subsection{The noiseless path D}
\label{sec:anytime-D}
Initialize $x_0=\xinit$ and $h_0=1$.
For $t=0,1,\ldots$, draw a fresh sample and set
\begin{equation}
 u_t=g(x_t;\xi_t),\qquad
 h_{t+1}=\sqrt{h_t^2+\norm{u_t}^2},\qquad
 x_{t+1}=x_t-\frac{u_t}{h_{t+1}}.
 \label{eq:D-path-update}
\end{equation}
This scalar AdaGrad-type normalization \cite{ward2019adagrad} uses one query per update.
Under bounded variance, its uniform report satisfies $\widetilde O(N^{-1/4})$ without bounded gradients \cite{wang2023adagrad}. Here, D supplies the noiseless rate, while S removes the horizon logarithms from the basic-regime guarantee. Appendix~\ref{sec:adagrad-comparison} details this distinction.
The point and accumulator persist throughout execution.
An independent reservoir makes the report $R_{D,N}$ uniform over the pre-update points $x_0,\ldots,x_{N-1}$ after $N$ queries.
The path runs in every regime.
Under exact gradients, Proposition~\ref{prop:noiseless-rate} gives
\[
 \E\norm{\nabla f(R_{D,N})}\le D_0N^{-1/2},\qquad
 D_0=2+2\Deltaf+2L\log(\max\{1,2L\}).
\]
Its finite-query reports also have finite expected gradient norm under noisy oracles, as required by validation.
Algorithm~\ref{alg:D-path} gives the update and reservoir operations.

\begin{algorithm}[t]
\caption{One query of the continuing D path}
\label{alg:D-path}
\begin{algorithmic}[1]
\Require State $(x,h,t,R)$, initially $(\xinit,1,0,\xinit)$
\State Save $y\gets x$ and query $u\gets g(y;\xi_t)$ with a fresh sample.
\State $h\gets\sqrt{h^2+\norm{u}^2}$ and $x\gets y-u/h$.
\State $t\gets t+1$
\State With probability $1/t$, set $R\gets y$ using an independent coin.
\State \Return report $R$, retaining the current state.
\end{algorithmic}
\end{algorithm}

\subsection{Validation and the anytime guarantee}
\label{sec:anytime-validation}
One global round executes one individual query on each path.
The first round includes the initial S and M queries and the first D update.
A shared-sample M update spans two rounds, retaining its sample and prescribed points between responses.
At rounds $n=1,2,4,\ldots$, freeze the current reports as $X_S,X_M,X_D$.
During rounds $n+1,\ldots,2n$, query each frozen point once per round to obtain a mean of $n$ fresh gradients.
At the end of round $2n$, publish the candidate with the smallest norm of the validation mean, breaking ties in the order S, M, D, before freezing the next reports.
The path states continue advancing while validation uses these fixed copies.
After each query response, complete any newly enabled query-free updates and publication before honoring an interruption.

\begin{algorithm}[t]
\caption{Anytime variance reduction with linear validation}
\label{alg:anytime}
\begin{algorithmic}[1]
\Require Initial point $\xinit$
\State Create S, M, D at $\xinit$ with their first queries pending; set $\widehat x\gets\xinit$.
\For{rounds $r=1,2,\ldots$}
 \State Execute one individual query on each path and complete any ready update.
 \If{a checkpoint $n$ is active}
  \State Query each frozen $X_S,X_M,X_D$ once and update its mean.
 \EndIf
 \If{$r=2n$ for the active checkpoint}
  \State Publish $\widehat x\gets X_{i_*}$, where $i_*=\arg\min_i\norm{\bar g_i}$.
 \EndIf
 \If{$r$ is a power of two}
  \State Set $n\gets r$, freeze the three reports, and start three means of size $n$.
 \EndIf
 \State After any individual query, return the last published $\widehat x$ if interrupted.
\EndFor
\end{algorithmic}
\end{algorithm}

\paragraph{Persistent state and query clocks.}
At block boundaries, S and M retain their endpoints, update their block parameters, and initialize their estimators from one fresh gradient. D retains its point and accumulator. Reservoir candidates and frozen validation points are separate copies, so publication never changes a path's current point.

At the end of global round $r$, each path has made exactly $r$ individual gradient evaluations, counting both responses of a shared-sample pair. The total budget $T$ also charges validation. Pauses retain pending samples, points, and running means, and return the last published candidate. Each path, its reservoir, and validation use independent random streams. A constant number of $d$-dimensional vectors suffices throughout execution.

\begin{theorem}
\label{thm:anytime}
Suppose Assumptions~\ref{ass:objective} and~\ref{ass:variance} hold.
For every deterministic integer query budget $T\ge1$, Algorithm~\ref{alg:anytime} returns the same point $\widehat x_T^{\rm lin}$ for all the following guarantees.
Set $P=1+\Deltaf+L^2+\sigma^2$, $P_0=1+\Deltaf+L^2$, and $P_{\rm av}=P+\bar L^2$ when average smoothness holds.
For a universal constant $C$,
\begin{equation}
 \E\norm{\nabla f(\widehat x_T^{\rm lin})}\le C
 \begin{cases}
 P^{5/2}T^{-1/4},&\text{under the basic assumptions},\\
 P_{\rm av}^{5/2}T^{-1/3},&\text{under average smoothness},\\
 P_0T^{-1/2},&\text{for noiseless gradients}.
 \end{cases}
 \label{eq:anytime-rates}
\end{equation}
The algorithm receives only $\xinit$.
All initialization and validation queries count toward $T$.
\end{theorem}
\paragraph{Remark.}
We have to point out that the fixed-budget method and its restarting variant have bounds with coefficients linear in the problem parameters, whereas the continuing-path bounds have higher-degree polynomial coefficients.

\subsection{Optional sublinear validation}
\label{sec:priority-extension}

The same paths support sublinear validation by a direct application of Lemma~\ref{lem:selection}.
At checkpoint $n$, freeze the reports $X_S,X_M,X_D$ and form
\begin{equation}
 m_n=\lceil n^{2/3}\rceil,\qquad
 \bar g_i=\frac1{m_n}\sum_{k=1}^{m_n}g(X_i;\xi_{i,k}^{\rm val}),
 \qquad i\in\{S,M,D\},
 \label{eq:validation-sizes}
\end{equation}
using fresh samples. Return the candidate with the smallest $\norm{\bar g_i}$,
breaking ties in the order S, M, D:
\begin{equation}
 Z_n=X_{\arg\min_{i\in\{S,M,D\}}\norm{\bar g_i}}.
 \label{eq:anytime-selection}
\end{equation}
Lemma~\ref{lem:selection} gives
\begin{equation}
 \E\norm{\nabla f(Z_n)}
 \le\min_{i\in\{S,M,D\}}\E\norm{\nabla f(X_i)}
       +\frac{\sqrt6\,\sigma}{\sqrt{m_n}}
 \le\min_i\E\norm{\nabla f(X_i)}+\sqrt6\,\sigma n^{-1/3}.
 \label{eq:sublinear-checkpoint}
\end{equation}
This error preserves both stochastic rates and vanishes for exact gradients.
The argument does not require independence among the candidates.
The batch exponent $2/3$ suffices because the faster stochastic target is $n^{-1/3}$.

During each of rounds $n+1,\ldots,n+m_n$, collect one validation gradient at each frozen candidate.
Since $m_n\le n$, all three means finish by round $2n$.
Publish $Z_n$ at that round, even if validation finished earlier.
Until publication, the returned point remains the preceding published candidate.
The optimization paths continue throughout and do not use validation samples.

\begin{algorithm}[!ht]
\caption{{Anytime variance reduction with sublinear validation}}
\label{alg:priority-anytime}
\begin{algorithmic}[1]
\Require Initial point $\xinit$
\State Create independent S, M, D states at $\xinit$ with their first queries pending.
\State Set the returned point $\widehat x\gets\xinit$.
\For{rounds $r=1,2,\ldots$}
 \State Execute one query on each path and complete any now-available updates.
 \State Add one query to each unfinished validation mean, if any.
 \If{$r=2n$ for the active checkpoint $n$}
  \State Publish $\widehat x\gets Z_n$ using \eqref{eq:anytime-selection}.
 \EndIf
 \If{$r$ is a power of two}
  \State Set $n\gets r$ and copy the path reports as $X_S,X_M,X_D$.
  \State Start one mean of size $m_n=\lceil n^{2/3}\rceil$ at each of $X_S,X_M,X_D$.
 \EndIf
 \State After any individual query, return the most recently published $\widehat x$ if interrupted.
\EndFor
\end{algorithmic}
\end{algorithm}
A round uses three optimization queries and at most three validation queries.
After each query, complete all available operations requiring no further queries, including any scheduled publication, before returning if interrupted.
The first publication occurs at query $9$, and before it the output is $\xinit$.
Every initialization, partial-pair, and validation query is charged.

\begin{corollary}[Sublinear validation]
\label{thm:priority-anytime}
Suppose Assumptions~\ref{ass:objective} and~\ref{ass:variance} hold.
For every deterministic integer query budget $T\ge1$, let $\widehat x_T^{\rm sub}$ be the output of Algorithm~\ref{alg:priority-anytime}.
Set $P=1+\Deltaf+L^2+\sigma^2$, $P_0=1+\Deltaf+L^2$, and, when average smoothness holds, $P_{\rm av}=P+\bar L^2$.
There is a universal numerical constant $C$ such that
\begin{equation}
 \E\norm{\nabla f(\widehat x_T^{\rm sub})}\le C
 \begin{cases}
 P^{5/2}T^{-1/4},&\text{under the basic assumptions},\\
 P_{\rm av}^{5/2}T^{-1/3},&\text{under average smoothness},\\
 P_0T^{-1/2},&\text{for noiseless gradients}.
 \end{cases}
 \label{eq:explicit-anytime-envelope}
\end{equation}
Cumulative validation work is $O(T^{2/3})$, including queries in unfinished batches.
\end{corollary}

The algorithm uses no numerical problem constants or total horizon.
Appendix~\ref{sec:anytime-selector} proves the query conversion and workload bound.
Both validation rules use the same path coefficients without multiplying them.
Linear validation retains its smaller $O(\sigma n^{-1/2})$ selection error, and
the sublinear rule uses only $3\lceil n^{2/3}\rceil$ validation queries per checkpoint.

\section{Conclusion}
\label{sec:conclusion}

We studied adaptive stochastic non-convex optimization without prior knowledge of smoothness constants, noise variance, the initial function gap, or the applicable oracle regime. Our fixed-budget construction returns one point with expected gradient-norm rates of $O(T^{-1/4})$ under objective smoothness and bounded variance, $O(T^{-1/3})$ under average smoothness, and $O(T^{-1/2})$ for exact gradients. Our anytime construction attains the same three rates at every deterministic query budget while preserving the current endpoints of its optimization paths. Its analysis combines adaptive-scale control, a uniform second-moment bound, and deterministic accounting for candidate reuse through unfinished blocks. An optional sample-mean selector further reduces cumulative validation work to $O(T^{2/3})$ while preserving these convergence exponents.
These results establish simultaneous adaptation across oracle conditions under an endpoint-preserving execution requirement. Improving the dependence on problem parameters and determining whether the same guarantees can be obtained with fewer optimization paths without logarithmic losses in the query budget remain directions for future work.

\label{page:main-end}
\clearpage
\appendix
\numberwithin{theorem}{section}
\numberwithin{lemma}{section}
\numberwithin{proposition}{section}
\numberwithin{corollary}{section}
\numberwithin{assumption}{section}
\section{Proof of Theorem~\ref{thm:simultaneous}}
\label{sec:proofs}

The proof controls each stage's candidate gradient and outgoing function gap in one descent inequality. The gap bound allows the next stage to begin at the preceding endpoint. Conditional estimator bounds then give the three candidate rates, and fresh validation transfers the best applicable rate to the common output.

For one stage, let $Y=x_0$ be its possibly random starting point, $Z=x_n$ its endpoint, and $X$ its uniform candidate.
All samples used in the stage are fresh.
Write
\[
 \zeta(x;\xi)=g(x;\xi)-\nabla f(x),\qquad
 e_t=v_t-\nabla f(x_t).
\]
Fresh noise has zero conditional mean, and hence
\begin{equation}
 \E\zeta(x_t;\xi_t)=0,\qquad
 \E\norm{\zeta(x_t;\xi_t)}^2\le\sigma^2.
 \label{eq:proof-fresh-noise}
\end{equation}
All expectations below include the randomness in $Y$.
The normalized updates give the deterministic distance bound
\[
 \max_{\text{queried }x}\norm{x-\xinit}
 \le n^{1/4}+n^{1/3}+n^{1/2},
\]
so smoothness bounds $f$ and $\nabla f$ at all queried points, and the variance assumption bounds the stochastic gradients' second moments.

\subsection{Uniform candidates and normalized descent}

The probability that the reservoir retains index $t$ is
\[
 \frac1{t+1}\prod_{s=t+1}^{n-1}\left(1-\frac1{s+1}\right)
 =\frac1{t+1}\prod_{s=t+1}^{n-1}\frac{s}{s+1}
 =\frac1n.
\]
The reservoir coins do not affect the iterates.
Therefore
\begin{equation}
 \E\norm{\nabla f(X)}
 =\frac1n\sum_{t=0}^{n-1}\E\norm{\nabla f(x_t)}.
 \label{eq:proof-uniform-candidate}
\end{equation}

\begin{lemma}[Normalized descent]
\label{lem:descent}
Under Assumption~\ref{ass:objective}, for any $x,v\in\R^d$ and $\eta\ge0$,
\begin{equation}
 f(x-\eta\N(v))
 \le f(x)-\eta\norm{\nabla f(x)}
       +2\eta\norm{v-\nabla f(x)}+\frac L2\eta^2.
 \label{eq:proof-one-step}
\end{equation}
Consequently, a stage with constant step size $\eta$ satisfies
\begin{equation}
 \begin{split}
 \E[f(Z)-f_*]+n\eta\E\norm{\nabla f(X)}
 \le{}&\E[f(Y)-f_*]\\
 &+2\eta\sum_{t=0}^{n-1}\E\norm{e_t}+\frac L2n\eta^2.
 \end{split}
 \label{eq:proof-stage-descent}
\end{equation}
\end{lemma}

\begin{proof}
For $v\ne0$, the triangle and Cauchy--Schwarz inequalities give
\begin{align*}
 \inner{\nabla f(x)}{\N(v)}
 &=\norm v+\inner{\nabla f(x)-v}{v/\norm v}\\
 &\ge\norm v-\norm{\nabla f(x)-v}\\
 &\ge\norm{\nabla f(x)}-2\norm{\nabla f(x)-v}.
\end{align*}
The last inequality also holds for $v=0$, since its two sides are $0$ and $-\norm{\nabla f(x)}$.
Thus
\begin{align*}
 f(x-\eta\N(v))
 &\le f(x)-\eta\inner{\nabla f(x)}{\N(v)}
              +\frac L2\eta^2\norm{\N(v)}^2\\
 &\le f(x)-\eta\norm{\nabla f(x)}
              +2\eta\norm{v-\nabla f(x)}+\frac L2\eta^2.
\end{align*}
Summing over the stage yields
\[
 f(x_n)-f_*+\eta\sum_{t=0}^{n-1}\norm{\nabla f(x_t)}
 \le f(Y)-f_*+2\eta\sum_{t=0}^{n-1}\norm{e_t}+\frac L2n\eta^2.
\]
Take expectations and use \eqref{eq:proof-uniform-candidate}.
\end{proof}

\subsection{Stage S}

\begin{lemma}
\label{lem:S}
Under Assumptions~\ref{ass:objective} and~\ref{ass:variance}, stage S with $\eta=n^{-3/4}$ and $\gamma=n^{-1/2}$ satisfies
\begin{equation}
 \E[f(Z_S)-f_*]+n^{1/4}\E\norm{\nabla f(X_S)}
 \le\E[f(Y)-f_*]+4\sigma+\frac52L.
 \label{eq:proof-S-bound}
\end{equation}
\end{lemma}

\begin{proof}
Subtracting $\nabla f(x_t)$ from \eqref{eq:S-recursion},
\[
 e_t=(1-\gamma)e_{t-1}+\gamma\zeta(x_t;\xi_t)
       +(1-\gamma)[\nabla f(x_{t-1})-\nabla f(x_t)].
\]
Since $e_0=\zeta(x_0;\xi_0)$, repeated substitution gives
\begin{equation}
 \begin{split}
 e_t={}&(1-\gamma)^t\zeta(x_0;\xi_0)
       +\gamma\sum_{i=1}^t(1-\gamma)^{t-i}\zeta(x_i;\xi_i)\\
 &+(1-\gamma)\sum_{i=1}^t(1-\gamma)^{t-i}
               [\nabla f(x_{i-1})-\nabla f(x_i)].
 \end{split}
 \label{eq:proof-S-expansion}
\end{equation}
For $i<j$, fresh sampling gives $\E\inner{\zeta(x_i;\xi_i)}{\zeta(x_j;\xi_j)}=0$.
The first line therefore has second moment at most
\begin{align*}
 \sigma^2\left[(1-\gamma)^{2t}
         +\gamma^2\sum_{j=0}^{t-1}(1-\gamma)^{2j}\right]
 &\le\sigma^2\left[(1-\gamma)^{2t}
                 +\frac{\gamma^2}{1-(1-\gamma)^2}\right]\\
 &=\sigma^2\left[(1-\gamma)^{2t}
                 +\frac{\gamma}{2-\gamma}\right]\\
 &\le\sigma^2[(1-\gamma)^{2t}+\gamma].
\end{align*}
For the second line, $\norm{x_i-x_{i-1}}\le\eta$ implies
\[
 \left\|(1-\gamma)\sum_{i=1}^t(1-\gamma)^{t-i}
            [\nabla f(x_{i-1})-\nabla f(x_i)]\right\|
 \le L\eta\sum_{j=0}^{t-1}(1-\gamma)^j
 \le\frac{L\eta}{\gamma}.
\]
Cauchy--Schwarz and $\sqrt{a+b}\le\sqrt a+\sqrt b$ now yield
\begin{equation}
 \E\norm{e_t}
 \le\sigma(1-\gamma)^t+\sigma\sqrt\gamma+\frac{L\eta}{\gamma}.
 \label{eq:proof-S-error}
\end{equation}
Substituting the parameters and summing,
\begin{align*}
 \eta\sum_{t=0}^{n-1}\E\norm{e_t}
 &\le\frac{\eta\sigma}{\gamma}
        +n\eta\sigma\sqrt\gamma+\frac{nL\eta^2}{\gamma}\\
 &=\sigma n^{-1/4}+\sigma+L
 \le2\sigma+L.
\end{align*}
Lemma~\ref{lem:descent} gives
\begin{align*}
 \E[f(Z_S)-f_*]+n^{1/4}\E\norm{\nabla f(X_S)}
 &\le\E[f(Y)-f_*]+4\sigma+2L+\frac L2n^{-1/2}\\
 &\le\E[f(Y)-f_*]+4\sigma+\frac52L.
\end{align*}
\end{proof}

\subsection{Stage M}

For fixed $x,y$, centering the paired difference gives
\begin{align}
 \E\norm{\zeta(x;\xi)-\zeta(y;\xi)}^2
 &=\E\norm{g(x;\xi)-g(y;\xi)}^2
       -\norm{\nabla f(x)-\nabla f(y)}^2\notag\\
 &\le\bar L^2\norm{x-y}^2.
 \label{eq:proof-centered-average}
\end{align}

\begin{lemma}
\label{lem:M}
Under Assumptions~\ref{ass:objective}, \ref{ass:variance}, and \ref{ass:average}, stage M with $\eta=\beta=n^{-2/3}$ satisfies
\begin{equation}
 \begin{split}
 \E[f(Z_M)-f_*]+n^{1/3}\E\norm{\nabla f(X_M)}
 \le{}&\E[f(Y)-f_*]+2\sigma\\
 &+2\sqrt{2(\bar L^2+\sigma^2)}+\frac L2n^{-1/3}.
 \end{split}
 \label{eq:proof-M-bound}
\end{equation}
\end{lemma}

\begin{proof}
The initial error satisfies $\E\norm{e_0}^2\le\sigma^2$.
For $t\ge1$, \eqref{eq:M-recursion} gives
\begin{align}
 e_t
 &=(1-\beta)e_{t-1}
       +\zeta(x_t;\xi_t)-(1-\beta)\zeta(x_{t-1};\xi_t)\notag\\
 &=(1-\beta)e_{t-1}
       +(1-\beta)[\zeta(x_t;\xi_t)-\zeta(x_{t-1};\xi_t)]
       +\beta\zeta(x_t;\xi_t).
 \label{eq:proof-M-error-recursion}
\end{align}
The new noise term has zero conditional mean.
Using $\norm{a+b}^2\le2\norm a^2+2\norm b^2$ for its two summands, \eqref{eq:proof-centered-average}, and $\norm{x_t-x_{t-1}}\le\eta$,
\begin{align*}
 \E\norm{e_t}^2
 &\le(1-\beta)^2\E\norm{e_{t-1}}^2
     +2(1-\beta)^2\bar L^2\eta^2+2\beta^2\sigma^2\\
 &\le(1-\beta)^2\E\norm{e_{t-1}}^2
     +2\beta^2(\bar L^2+\sigma^2).
\end{align*}
Iterating the recursion and summing the geometric series,
\begin{align*}
 \E\norm{e_t}^2
 &\le(1-\beta)^{2t}\sigma^2
      +2\beta^2(\bar L^2+\sigma^2)
             \sum_{j=0}^{t-1}(1-\beta)^{2j}\\
 &\le(1-\beta)^{2t}\sigma^2
      +\frac{2\beta^2(\bar L^2+\sigma^2)}{1-(1-\beta)^2}\\
 &=(1-\beta)^{2t}\sigma^2
      +\frac{2\beta(\bar L^2+\sigma^2)}{2-\beta}\\
 &\le(1-\beta)^{2t}\sigma^2+2\beta(\bar L^2+\sigma^2).
\end{align*}
Thus
\begin{equation}
 \E\norm{e_t}
 \le\sigma(1-\beta)^t
       +\sqrt{2(\bar L^2+\sigma^2)}\,n^{-1/3},
 \label{eq:proof-M-error}
\end{equation}
and
\begin{align*}
 \eta\sum_{t=0}^{n-1}\E\norm{e_t}
 &\le\frac{\eta\sigma}{\beta}
       +n\eta\sqrt{2(\bar L^2+\sigma^2)}\,n^{-1/3}\\
 &=\sigma+\sqrt{2(\bar L^2+\sigma^2)}.
\end{align*}
Substitute this bound, $n\eta=n^{1/3}$, and $n\eta^2=n^{-1/3}$ into \eqref{eq:proof-stage-descent}.
\end{proof}

\subsection{The noiseless stages}

For a noiseless oracle, $v_0=\nabla f(x_0)$ in M and its recursion gives
\[
 v_t-\nabla f(x_t)=(1-\beta)[v_{t-1}-\nabla f(x_{t-1})]=0
\]
by induction. Hence Lemma~\ref{lem:descent}, with $\eta=n^{-2/3}$, implies
\begin{equation}
 \E[f(Z_M)-f_*]\le\E[f(Y)-f_*]+\frac L2n^{-1/3}.
 \label{eq:proof-M-noiseless-gap}
\end{equation}

\begin{lemma}[Fresh-gradient stage under exact gradients]
\label{lem:D}
For a noiseless oracle, stage D with $\eta=n^{-1/2}$ satisfies
\begin{equation}
 \E[f(Z_D)-f_*]+\sqrt n\E\norm{\nabla f(X_D)}
 \le\E[f(Y)-f_*]+\frac L2.
 \label{eq:proof-D-bound}
\end{equation}
\end{lemma}
\begin{proof}
Equation~\eqref{eq:D-direction} gives $e_t=0$ at every update.
Substitute this identity, $n\eta=\sqrt n$, and $n\eta^2=1$ into
\eqref{eq:proof-stage-descent} and take the expectation over the incoming point and reservoir coins.
\end{proof}

\subsection{Candidate selection}

\begin{proof}[Proof of Lemma~\ref{lem:selection}]
Write $e_i=\widehat g_i-\nabla f(X_i)$.
Since the validation samples are fresh,
\[
 \E\norm{e_i}^2
 =\frac1{n^2}\sum_{s=1}^n
   \E\norm{g(X_i;\xi_{i,s})-\nabla f(X_i)}^2
 \le\frac{\sigma^2}{n}.
\]
Let $\widehat i$ be the selected index and let $i_*\in\arg\min_i\norm{\nabla f(X_i)}$.
Then
\begin{align*}
 \norm{\nabla f(X_{\widehat i})}
 &\le\norm{\widehat g_{\widehat i}}+\norm{e_{\widehat i}}\\
 &\le\norm{\widehat g_{i_*}}+\norm{e_{\widehat i}}\\
 &\le\norm{\nabla f(X_{i_*})}+\norm{e_{i_*}}+\norm{e_{\widehat i}}.
\end{align*}
If $\widehat i=i_*$, the excess true gradient norm is zero.
Otherwise the two indices are distinct and
\[
 \norm{e_{i_*}}+\norm{e_{\widehat i}}
 \le\sqrt{2(\norm{e_{i_*}}^2+\norm{e_{\widehat i}}^2)}
 \le\sqrt{2\sum_i\norm{e_i}^2}.
\]
In either case, Jensen's inequality gives
\begin{align*}
 \E\norm{\nabla f(X_{\widehat i})}
 &\le\E\min_i\norm{\nabla f(X_i)}
           +\sqrt{2\sum_i\E\norm{e_i}^2}\\
 &\le\min_i\E\norm{\nabla f(X_i)}+\sqrt{6\sigma^2/n}.
\end{align*}
\end{proof}

\subsection{Combining the stages}

\begin{proof}[Proof of Theorem~\ref{thm:simultaneous}]
\textit{Standard smoothness.} Lemma~\ref{lem:S}, $f(Z_S)\ge f_*$, and Lemma~\ref{lem:selection} imply
\begin{align*}
 \E\norm{\nabla f(X_S)}
 &\le(\Deltaf+4\sigma+\tfrac52L)n^{-1/4},\\
 \E\norm{\nabla f(\widehat x_T)}
 &\le(\Deltaf+4\sigma+\tfrac52L)n^{-1/4}
                 +\sqrt6\,\sigma n^{-1/2}\\
 &\le(\Deltaf+7\sigma+3L)n^{-1/4}.
\end{align*}
Later stages leave the stored point $X_S$ unchanged.

\textit{Average smoothness.} The outgoing-gap part of Lemma~\ref{lem:S} gives
\begin{equation}
 \E[f(Z_S)-f_*]\le\Deltaf+4\sigma+\frac52L.
 \label{eq:proof-S-final-gap}
\end{equation}
Apply Lemma~\ref{lem:M} with $Y=Z_S$ to obtain
\begin{align*}
 n^{1/3}\E\norm{\nabla f(X_M)}
 &\le\E[f(Z_S)-f_*]+2\sigma
       +2\sqrt{2(\bar L^2+\sigma^2)}+\frac L2n^{-1/3}\\
 &\le\Deltaf+6\sigma+3L+2\sqrt{2(\bar L^2+\sigma^2)}.
\end{align*}
Since $\sqrt{\bar L^2+\sigma^2}\le\bar L+\sigma$,
\begin{align*}
 \E\norm{\nabla f(\widehat x_T)}
 &\le\bigl[\Deltaf+(6+\sqrt6)\sigma+3L
             +2\sqrt{2(\bar L^2+\sigma^2)}\bigr]n^{-1/3}\\
 &\le(\Deltaf+3L+3\bar L+12\sigma)n^{-1/3}.
\end{align*}

\textit{Noiseless gradients.}
Lemma~\ref{lem:S} with $\sigma=0$, followed by
\eqref{eq:proof-M-noiseless-gap}, gives
\[
 \E[f(Z_M)-f_*]\le\Deltaf+\frac52L+\frac L2n^{-1/3}
 \le\Deltaf+3L.
\]
Apply Lemma~\ref{lem:D} with $Y=Z_M$.
Validation is exact, so
\[
 \E\norm{\nabla f(\widehat x_T)}
 \le\E\norm{\nabla f(X_D)}
 \le(\Deltaf+\tfrac72L)n^{-1/2}
 \le(\Deltaf+4L)n^{-1/2}.
\]

\textit{Evaluation count.} S and D use $n$ gradients each, M uses $1+2(n-1)$, and validation uses $3n$.
Thus
\[
 n+(2n-1)+n+3n=7n-1\le T.
\]
For $n=1$, all sums over $t\ge1$ are empty, and the same bounds hold with $\gamma=\beta=1$.
Finally, $T/7<n+1\le2n$ gives $n\ge T/14$.

\end{proof}

\section{Estimation error within an S or M block}
\label{sec:block-estimation}

We use Assumptions~\ref{ass:objective} and~\ref{ass:variance}, with Assumption~\ref{ass:average} added for M.
Within one block, $a,b,\eta$ are fixed at their starting values.
Let $\E_j$ denote expectation conditional on the history at the start of block $j$.
We omit the block index from its iterates and write
\[
 e_t=v_t-\nabla f(x_t),\qquad 1\le t\le b.
\]
All block lengths are even.
For later use, set $\mathcal E=4\sqrt{\sigma^2+\Lambda^2}$, where $\Lambda=L$ for S and $\Lambda=\bar L$ for M.
This constant is used only in the proof.

\begin{lemma}[Mean error and second-half maximum]
\label{lem:block-errors}
For every block of Algorithm~\ref{alg:block-path}, under the assumptions stated above,
\begin{equation}
 \E_j\left(\frac1b\sum_{t=1}^b\norm{e_t}\right)^2
 \le\frac{2\mathcal E^2}{b},\qquad
 \E_j\max_{b/2<t\le b}\norm{e_t}^2
 \le\frac{2\mathcal E^2}{b}.
 \label{eq:block-errors}
\end{equation}
The conditioning precedes the block's success or failure.
\end{lemma}

\begin{proof}
Every update has length at most $\eta$.
In particular,
\begin{equation}
 \norm{x_t-x_i}\le\sum_{k=i}^{t-1}\norm{x_{k+1}-x_k}
 \le(t-i)\eta,\qquad i<t.
 \label{eq:block-distance}
\end{equation}
We prove the two bounds separately for S and M.

\paragraph{Path S.}
Put $\zeta_i=g(x_i;\xi_i)-\nabla f(x_i)$ and $M_t=\sum_{i=1}^t\zeta_i$.
Each $x_i$ is chosen before its fresh sample.
Thus $\E[\zeta_i\mid\text{history before }\xi_i]=0$, and different increments have zero expected inner product.
Consequently,
\begin{align*}
 \E_j\norm{M_t}^2
 &=\sum_{i=1}^t\E_j\norm{\zeta_i}^2
    +2\sum_{i<k}\E_j\inner{\zeta_i}{\zeta_k}\\
 &=\sum_{i=1}^t\E_j\norm{\zeta_i}^2
 \le t\sigma^2.
\end{align*}
Expanding the averaging recursion gives
\begin{align*}
 v_t&=\frac1t\sum_{i=1}^t g(x_i;\xi_i),\\
 e_t&=\frac{M_t}{t}
      +\frac1t\sum_{i=1}^t\bigl(\nabla f(x_i)-\nabla f(x_t)\bigr),\\
 \left\|\frac1t\sum_{i=1}^t\bigl(\nabla f(x_i)-\nabla f(x_t)\bigr)\right\|
 &\le\frac{L\eta}{t}\sum_{i=1}^t(t-i)
 =\frac{L\eta(t-1)}2.
\end{align*}
The noise term is a martingale average, and the other term measures movement within the block.
By the triangle inequality in $L^2$ and $\sum_{t=1}^b t^{-1/2}\le2\sqrt b$,
\begin{align*}
 \left[\E_j\left(\frac1b\sum_{t=1}^b\norm{e_t}\right)^2\right]^{1/2}
 &\le\frac1b\sum_{t=1}^b(\E_j\norm{e_t}^2)^{1/2}\\
 &\le\frac{\sigma}{b}\sum_{t=1}^bt^{-1/2}+\frac{L\eta b}{2}\\
 &\le\frac{2\sigma}{\sqrt b}+\frac{L}{2a}\sqrt{\frac2b}\\
 &\le\left(\sqrt2\sigma+\frac L2\right)\sqrt{\frac2b}
 \le\mathcal E\sqrt{\frac2b}.
\end{align*}
We used $b\eta=\sqrt{2/b}/a$ from the S step length.

For the second-half maximum, $1/t\le2/b$ whenever $t>b/2$.
The norm of a vector martingale is a nonnegative submartingale, so Doob's $L^2$ inequality applies to this finite square-integrable sequence and gives
\[
 \E_j\max_{t\le b}\norm{M_t}^2\le4\E_j\norm{M_b}^2\le4b\sigma^2.
\]
It follows that
\begin{align*}
 \left[\E_j\max_{b/2<t\le b}\norm{e_t}^2\right]^{1/2}
 &\le\frac2b\left[\E_j\max_{t\le b}\norm{M_t}^2\right]^{1/2}
       +\frac{L\eta b}{2}\\
 &\le\frac{4\sigma}{\sqrt b}+\frac{L}{2a}\sqrt{\frac2b}\\
 &\le\left(2\sqrt2\sigma+\frac L2\right)\sqrt{\frac2b}
 \le\mathcal E\sqrt{\frac2b}.
\end{align*}
For the last inequality, Cauchy--Schwarz gives $(2\sqrt2\sigma+L/2)^2\le(8+1/4)(\sigma^2+L^2)\le\mathcal E^2$.
Squaring proves both S bounds.

\paragraph{Path M.}
Write $\zeta_t(z)=g(z;\xi_t)-\nabla f(z)$.
Subtract the true gradient from \eqref{eq:block-estimators} and multiply by $t$ to obtain
\begin{align*}
 te_t
 &=(t-1)e_{t-1}+\zeta_t(x_t)
       +(t-1)\bigl(\zeta_t(x_t)-\zeta_t(x_{t-1})\bigr).
\end{align*}
Hence $M_t=te_t$ is a vector martingale with $M_1=\zeta_1(x_1)$.
Both points of the fresh pair are fixed before sampling.
Centering its difference cannot increase the second moment, since
\begin{align*}
 &\E\bigl[\norm{\zeta_t(x_t)-\zeta_t(x_{t-1})}^2\mid\text{past}\bigr]\\
 &\quad=\E\bigl[\norm{g(x_t;\xi_t)-g(x_{t-1};\xi_t)}^2\mid\text{past}\bigr]
          -\norm{\nabla f(x_t)-\nabla f(x_{t-1})}^2\\
 &\quad\le\bar L^2\norm{x_t-x_{t-1}}^2
 \le\bar L^2\eta^2.
\end{align*}
The martingale increment therefore has conditional second moment at most $2\sigma^2+2(t-1)^2\bar L^2\eta^2$.
Orthogonality of increments yields
\begin{align}
 \E_j\norm{M_t}^2
 &\le\sigma^2+\sum_{i=2}^t\bigl(2\sigma^2+2(i-1)^2\bar L^2\eta^2\bigr)\notag\\
 &\le2\sigma^2t+2\bar L^2\eta^2\frac{(t-1)t(2t-1)}6\notag\\
 &\le2\sigma^2t+\frac23\bar L^2\eta^2t^3.
 \label{eq:M-block-martingale}
\end{align}
Since $e_t=M_t/t$ and $b\eta=1/a$, the mean-error bound follows from
\begin{align*}
 \left[\E_j\left(\frac1b\sum_{t=1}^b\norm{e_t}\right)^2\right]^{1/2}
 &\le\frac1b\sum_{t=1}^b
     \left(\frac{\sqrt2\sigma}{\sqrt t}+\sqrt{\frac23}\bar L\eta\sqrt t\right)\\
 &\le\frac{2\sqrt2\sigma}{\sqrt b}+\sqrt{\frac23}\bar L\eta\sqrt b\\
 &=\left(2\sigma+\frac{\bar L}{a\sqrt3}\right)\sqrt{\frac2b}
 \le\mathcal E\sqrt{\frac2b}.
\end{align*}
Here $(2\sigma+\bar L/\sqrt3)^2\le(4+1/3)(\sigma^2+\bar L^2)\le\mathcal E^2$.
Doob's inequality gives the second-half bound directly,
\begin{align*}
 \E_j\max_{b/2<t\le b}\norm{e_t}^2
 &\le\frac4{b^2}\E_j\max_{t\le b}\norm{M_t}^2\\
 &\le\frac{16}{b^2}\E_j\norm{M_b}^2\\
 &\le\frac{32\sigma^2}{b}+\frac{32}{3}\bar L^2\eta^2b\\
 &=\frac{32}{b}\left(\sigma^2+\frac{\bar L^2}{3a^2}\right)
 \le\frac{2\mathcal E^2}{b}.
\end{align*}
This completes both cases.
\end{proof}

\section{Convergence of the continuing S and M paths}
\label{sec:block-convergence}

We retain the assumptions of Appendix~\ref{sec:block-estimation}.
The proof couples the block-scale potential to function decrease to bound $\E a_\infty^2$ uniformly over local query counts. It then bounds the number of queries that reuse each failed block's candidate, including queries in an unfinished block. Applying this deterministic reuse bound before taking expectations avoids conditioning the estimator-error bound on block failure. H\"older's inequality then gives the report rates.
In this appendix only, set $p=2$ for S and $p=1$ for M to write their two update rules together.
This index is not an additional parameter of the algorithm.
Let $(x_j,a_j,b_j)$ be the state at the start of block $j$, so $x_1=\xinit$, $a_1=1$, and $b_1=2$.
If $m_j$ estimates in that block are large, Table~\ref{tab:block-updates} reads
\begin{equation}
 \begin{aligned}
 a_{j+1}-a_j&=\left(\frac2{b_j}\right)^{p/2}\frac{m_j}{b_j},\\
 b_{j+1}&=\begin{cases}
 \max\{b_j/2^{2/p},2\},&m_j\ge b_j/2,\\
 2^{2/p}b_j,&m_j<b_j/2,
 \end{cases}
 \qquad
 \eta_j=\frac{(2/b_j)^{(p-1)/2}}{a_jb_j}.
 \end{aligned}
 \label{eq:block-unified}
\end{equation}
All constants denoted by $C$ below are finite functions only of $(\Deltaf,L,\sigma,\Lambda)$, with $\Lambda=L$ for S and $\Lambda=\bar L$ for M.
They may change between displays and never depend on the horizon or dimension.

\begin{proposition}[Reports at individual query counts]
\label{prop:block-rates}
For every deterministic integer $N\ge1$, the S and M reports satisfy
\begin{equation}
 \E\norm{\nabla f(R_{S,N})}=O(N^{-1/4}),\qquad
 \E\norm{\nabla f(R_{M,N})}=O(N^{-1/3}).
 \label{eq:block-rates}
\end{equation}
The first bound uses objective smoothness and bounded variance.
The second additionally uses average smoothness.
The bounds include interruptions inside blocks and paired queries.
\end{proposition}

\subsection{A potential for the block parameters}

The quantity
\[
 u_j=\frac{(2/b_j)^{p/2}}{a_j}
\]
measures both the scale increment relative to $a_j$ and the expected error charge of a block.
It is useful only for the analysis.
Because $a_j\ge1$ and $b_j\ge2$, we have $0<u_j\le1$.
The scale update gives
\[
 \frac{a_{j+1}}{a_j}=1+u_j\frac{m_j}{b_j}.
\]
If the block succeeds, then $m_j/b_j\ge1/2$ and
\begin{align*}
 \log\frac{a_{j+1}}{a_j}
 &\ge\log(1+u_j/2)
 \ge\frac{u_j/2}{1+u_j/2}
 =\frac{u_j}{2+u_j}
 \ge\frac{u_j}{3},\\
 u_{j+1}
 &=\frac{(2/b_{j+1})^{p/2}}{a_{j+1}}
 \le\frac{2(2/b_j)^{p/2}}{a_j}
 =2u_j.
\end{align*}
The logarithm inequality follows by integrating $1/(1+s)$ over $0\le s\le u_j/2$.
These two estimates imply
\[
 9\log(a_{j+1}/a_j)+2(u_j-u_{j+1})\ge3u_j-2u_j=u_j.
\]
If the block fails, then $u_{j+1}\le u_j/2$ and $a_{j+1}\ge a_j$, so
\[
 9\log(a_{j+1}/a_j)+2(u_j-u_{j+1})\ge0+u_j.
\]
Summing over either type of block and using $u_1=a_1=1$ gives
\begin{equation}
 H_J:=\sum_{j=1}^Ju_j
 \le9\log a_{J+1}+2(u_1-u_{J+1})
 \le9\log a_{J+1}+2.
 \label{eq:block-potential}
\end{equation}
Thus only the growth of $a_j$ remains to be controlled.

\subsection{Function decrease pays for scale growth}

Within block $j$, let $x_{j,t}$ be the point before update $t$ and let $e_{j,t}=v_{j,t}-\nabla f(x_{j,t})$.
For $v\ne0$,
\[
 \inner{\nabla f(x)}{\N(v)}
 =\norm v-\inner{v-\nabla f(x)}{\N(v)}
 \ge\norm v-\norm{v-\nabla f(x)}.
\]
For $v=0$ the same inequality holds because its right side is nonpositive.
Let $I_{j,t}$ indicate the event $\norm{v_{j,t}}>4a_j\sqrt{2/b_j}$.
Smoothness then gives
\begin{align*}
 f(x_{j,t+1})-f(x_{j,t})
 &\le-\eta_j\norm{v_{j,t}}+\eta_j\norm{e_{j,t}}+\frac L2\eta_j^2\\
 &\le-4a_j\sqrt{2/b_j}\,\eta_j I_{j,t}
       +\eta_j\norm{e_{j,t}}+\frac L2\eta_j^2.
\end{align*}
The identities
\[
 a_j\sqrt{2/b_j}\,\eta_jm_j=a_{j+1}-a_j,
 \qquad
 b_j\eta_j^2=\frac{u_j^2}{2}
\]
follow by substituting \eqref{eq:block-unified}.
Define the nonnegative error charge of the block by
\[
 Z_j=\eta_j\sum_{t=1}^{b_j}\norm{e_{j,t}}.
\]
Summing the stepwise inequality gives
\begin{equation}
 f(x_{j+1})-f(x_j)
 \le-4(a_{j+1}-a_j)+Z_j+\frac L2u_j^2.
 \label{eq:block-energy}
\end{equation}
We enlarged the curvature term from $Lu_j^2/4$ to $Lu_j^2/2$.
Lemma~\ref{lem:block-errors} and $b_j\eta_j\sqrt{2/b_j}=u_j$ imply
\begin{align*}
 \E_jZ_j^2
 &=(b_j\eta_j)^2\E_j\left(\frac1{b_j}\sum_t\norm{e_{j,t}}\right)^2
 \le\mathcal E^2u_j^2,\\
 \E_jZ_j&\le(\E_jZ_j^2)^{1/2}\le\mathcal E u_j.
\end{align*}
Set $c=\mathcal E+L/2$ and center the accumulated error,
\[
 \mathcal M_J=\sum_{j=1}^J(Z_j-\E_jZ_j).
\]
The summands are martingale differences at block boundaries.
Since $u_j^2\le u_j$ and $f(x_{J+1})\ge f_*$, summing \eqref{eq:block-energy} gives
\begin{align}
 4(a_{J+1}-1)
 &\le\Deltaf+cH_J+\mathcal M_J,\notag\\
 \E\mathcal M_J&=0,\qquad
 \E\mathcal M_J^2
 =\sum_{j=1}^J\E(Z_j-\E_jZ_j)^2
 \le\mathcal E^2\E H_J.
 \label{eq:block-scale-martingale}
\end{align}

These finite-prefix quantities are integrable before we use any uniform bound.
Indeed, $a_j\le j$, $b_j\le2(2^{2/p})^{j-1}$, and
\[
 b_j\eta_j=\frac{(2/b_j)^{(p-1)/2}}{a_j}\le1.
\]
The first $J$ blocks therefore contain a deterministically bounded number of queries, and every queried point is within distance $J$ of $\xinit$.
Smoothness bounds the gradients and function values on this ball.
The variance assumption then bounds the second moments of the fresh gradient errors.
This justifies the finite martingale expansions and expectations above.

\subsection{A uniform second-moment bound for the scale}

Combining \eqref{eq:block-potential} and \eqref{eq:block-scale-martingale},
\[
 a_{J+1}
 \le1+\frac{\Deltaf+2c}{4}+\frac{9c}{4}\log a_{J+1}+\frac{\mathcal M_J}{4}.
\]
For any $a>0$ and $r>0$, concavity of the logarithm gives
\[
 \log a\le\log r+\frac{a-r}{r}\le\log r+\frac ar.
\]
Use this inequality at $r=\max\{1,9c/2\}$, so that $(9c/4)/r\le1/2$.
After moving the resulting $a_{J+1}/2$ term to the left, we obtain
\begin{equation}
 1\le a_{J+1}\le K+\frac{\mathcal M_J}{2},\qquad
 K=2+\frac{\Deltaf}{2}+c+\frac{9c}{2}\log\bigl(\max\{1,9c/2\}\bigr).
 \label{eq:block-constant-K}
\end{equation}
This single constant is independent of $J$ and dimension.
Taking expectations, followed by Jensen's inequality in \eqref{eq:block-potential}, yields
\begin{align*}
 \E a_{J+1}&\le K,\\
 \E H_J&\le9\E\log a_{J+1}+2
 \le9\log(\E a_{J+1})+2
 \le9\log K+2,\\
 \E\mathcal M_J^2&\le\mathcal E^2(9\log K+2).
\end{align*}
The upper bound in \eqref{eq:block-constant-K} is nonnegative because it bounds $a_{J+1}\ge1$ pointwise.
It can therefore be squared without reversing the inequality.
Using $\E\mathcal M_J=0$ gives
\begin{align*}
 \E a_{J+1}^2
 &\le\E(K+\mathcal M_J/2)^2\\
 &=K^2+K\E\mathcal M_J+\tfrac14\E\mathcal M_J^2\\
 &\le K^2+\tfrac14\mathcal E^2(9\log K+2).
\end{align*}
The sequences $a_{J+1}$, $a_{J+1}^2$, and $H_J$ are nonnegative and nondecreasing.
Monotone convergence proves
\begin{equation}
 \begin{gathered}
 a_j\uparrow a_\infty<\infty\quad\text{almost surely},\qquad
 \E a_\infty^2\le K^2+\tfrac14\mathcal E^2(9\log K+2),\\
 H_\infty:=\sum_{j\ge1}u_j\le9\log a_\infty+2,
 \qquad \E H_\infty\le9\log K+2.
 \end{gathered}
 \label{eq:block-scale-moments}
\end{equation}
Finiteness almost surely follows from the finite second moment.
The pathwise bound on $H_\infty$ follows by taking limits in \eqref{eq:block-potential}.

\subsection{A bound on the gradients visited by the path}

We also need to control the labels used by blocks of length two.
Dropping the nonpositive descent term from the stepwise inequality bounds any prefix, including a point inside a block.
For every visited point $z$,
\[
 f(z)-f_*\le\Deltaf+\sum_{j\ge1}Z_j+\frac L2H_\infty.
\]
All summands $Z_j$ are nonnegative, so Tonelli's theorem and the tower property give
\[
 \E\sum_{j\ge1}Z_j
 =\sum_{j\ge1}\E\E_jZ_j
 \le\mathcal E\sum_{j\ge1}\E u_j
 =\mathcal E\E H_\infty.
\]
For any point $z$, smoothness applied at $z-\nabla f(z)/L$ implies
\begin{align*}
 f_*&\le f(z-\nabla f(z)/L)\le f(z)-\frac{\norm{\nabla f(z)}^2}{2L},\\
 \norm{\nabla f(z)}^2&\le2L(f(z)-f_*).
\end{align*}
Let $G_\infty$ be the supremum of gradient norms over the countable sequence of iterates.
Combining the last three displays gives
\begin{equation}
 \E G_\infty^2
 \le2L\bigl(\Deltaf+c(9\log K+2)\bigr)<\infty.
 \label{eq:block-gradient-moment}
\end{equation}
This estimate is derived from the algorithm and the stated assumptions.
No bounded-gradient assumption has been introduced.

\subsection{Labels from blocks of length two}

A block of length two uses its starting point as the label and costs at most three queries.
Every such block after the first is preceded by a distinct successful block of length $2$ or $2\cdot2^{2/p}$.
In either case that predecessor increases $a$ by at least $1/4$, because
\[
 a_{j+1}-a_j
 =\left(\frac2{b_j}\right)^{p/2}\frac{m_j}{b_j}
 \ge\frac12\cdot\frac12.
\]
The number of length-two blocks is at most $1+4(a_\infty-1)\le4a_\infty$.
Write $\ell_r$ for the label of local query $r$.
Counting a partially completed block in full and using Cauchy--Schwarz gives
\begin{equation}
 \E\sum_{r:\,b=2}\norm{\nabla f(\ell_r)}
 \le12\E[a_\infty G_\infty]
 \le12\sqrt{\E a_\infty^2\,\E G_\infty^2}
 \le C.
 \label{eq:coarse-reports}
\end{equation}
Both second moments are finite by \eqref{eq:block-scale-moments} and \eqref{eq:block-gradient-moment}.

\subsection{Labels retained from failed blocks}

For block $s$, let $D_s=\max_{b_s/2<t\le b_s}\norm{e_{s,t}}$.
If it fails, its recorded point satisfies
\begin{equation}
 \norm{\nabla f(Y_s)}\le4a_s\sqrt{2/b_s}+D_s,
 \qquad
 \E_sD_s\le\mathcal E\sqrt{2/b_s}.
 \label{eq:failed-report}
\end{equation}
The conditional error bound holds without conditioning on failure.

After this failure, the next block has length $2^{2/p}b_s$.
Until the next failed block finishes, the lengths greater than two form a decreasing geometric sequence.
Every block has query cost between $b$ and $2b$.
The number of queries with $b>2$ that can ever receive label $Y_s$ is therefore at most
\begin{equation}
 2\sum_{k=0}^{\infty}2^{2/p}b_s\,2^{-2k/p}
 =\frac{2\cdot2^{2/p}}{1-2^{-2/p}}b_s
 \le16b_s.
 \label{eq:report-reuse}
\end{equation}
This count includes an unfinished final block.
After length two is reached, labels are starting points until another failure finishes, and those queries were already covered by \eqref{eq:coarse-reports}.

Let $Q_s$ be the local query count when block $s$ finishes.
Its length and query cost are fixed at its start, so the event $\{Q_s\le N\}$ is measurable at that start for each deterministic $N$.
First apply the pathwise reuse bound, then remove the failure indicator from a nonnegative sum, and only then take the conditional expectation of $D_s$.
These steps give
\begin{align}
 \E\sum_{\substack{r\le N\\b>2}}\norm{\nabla f(\ell_r)}
 &\le16\E\sum_{s\ge1}\mathbf1_{\{Q_s\le N\}}
       \mathbf1_{\{s\text{ failed}\}}b_s(4a_s\sqrt{2/b_s}+D_s)\notag\\
 &\le16\E\sum_{s\ge1}\mathbf1_{\{Q_s\le N\}}
       b_s(4a_s\sqrt{2/b_s}+D_s)\notag\\
 &=16\E\sum_{s\ge1}\mathbf1_{\{Q_s\le N\}}
       b_s(4a_s\sqrt{2/b_s}+\E_sD_s)\notag\\
 &\le16\sqrt2\E\sum_{s:Q_s\le N}(4a_s+\mathcal E)\sqrt{b_s}.
 \label{eq:fine-reports}
\end{align}
Tonelli's theorem applies to each nonnegative sum.
In particular, we have not assumed that the error is independent of failure or of later reuse.

To bound the last sum, put $\gamma=1/(p+2)$.
For completed blocks, query counting and the potential give
\[
 \sum_{Q_s\le N}b_s\le N,
 \qquad
 \sum_{Q_s\le N}u_s\le H_\infty.
\]
The exponent identity $1-\gamma-p\gamma/2=1/2$ implies
\[
 \sqrt{b_s}=2^{-p\gamma/2}b_s^{1-\gamma}u_s^\gamma a_s^\gamma.
\]
H\"older's inequality with exponents $1/(1-\gamma)$ and $1/\gamma$ therefore yields
\begin{align*}
 \sum_{Q_s\le N}\sqrt{b_s}
 &\le a_\infty^\gamma
       \sum_{Q_s\le N}b_s^{1-\gamma}u_s^\gamma\\
 &\le a_\infty^\gamma
       \left(\sum_{Q_s\le N}b_s\right)^{1-\gamma}
       \left(\sum_{Q_s\le N}u_s\right)^\gamma\\
 &\le a_\infty^\gamma N^{1-\gamma}H_\infty^\gamma.
\end{align*}
We discarded the factor $2^{-p\gamma/2}\le1$ in the first line.
From $\log a\le a-1$ and \eqref{eq:block-scale-moments}, $H_\infty\le9a_\infty$.
Since $a_\infty\ge1$ and $1+2\gamma\le2$ for both $p=1,2$,
\begin{align*}
 (4a_\infty+\mathcal E)a_\infty^\gamma H_\infty^\gamma
 &\le9^\gamma(4+\mathcal E)a_\infty^{1+2\gamma}\\
 &\le9^\gamma(4+\mathcal E)a_\infty^2.
\end{align*}
Substitution into \eqref{eq:fine-reports} proves
\begin{equation}
 \E\sum_{\substack{r\le N\\b>2}}\norm{\nabla f(\ell_r)}
 \le16\sqrt2\,9^\gamma(4+\mathcal E)\E a_\infty^2\,N^{1-\gamma}
 \le C N^{1-\gamma}.
 \label{eq:fine-final}
\end{equation}

\subsection{Proof of Proposition~\ref{prop:block-rates}}

The reservoir coins are independent of all labels and never affect optimization.
For any fixed label index $r\le N$, its probability of being retained is
\[
 \frac1r\prod_{k=r+1}^N\left(1-\frac1k\right)
 =\frac1r\prod_{k=r+1}^N\frac{k-1}{k}
 =\frac1N.
\]
Conditional on the labels, the candidate is therefore uniform.
Using \eqref{eq:coarse-reports} and \eqref{eq:fine-final},
\begin{align*}
 \E\norm{\nabla f(R_N)}
 &=\frac1N\E\sum_{r=1}^N\norm{\nabla f(\ell_r)}\\
 &\le\frac C N+CN^{-\gamma}
 \le C N^{-\gamma}.
\end{align*}
The exponent is $1/4$ for S and $1/3$ for M.
The count includes partial blocks by \eqref{eq:coarse-reports} and \eqref{eq:report-reuse}, and both labels of a pair are processed separately.
All constants are bounded explicitly by \eqref{eq:block-constant-K}--\eqref{eq:block-gradient-moment} and the numerical coefficients above.
This proves the proposition with the stated parameter dependence.

\section{The continuing gradient-norm path in the noiseless regime}
\label{sec:noiseless-path}
This appendix gives a direct analysis of path D, a scalar AdaGrad-Norm update
\cite{ward2019adagrad}.
The path uses a single gradient evaluation per update and does not receive
the smoothness constant or any label specifying the oracle regime.

Initialize $x_0=\xinit$ and $h_0=1$.
At local query $t+1$, for $t=0,1,\ldots$, draw a fresh independent sample,
evaluate $u_t=g(x_t;\xi_t)$, and set
\begin{equation}
 h_{t+1}=\sqrt{h_t^2+\norm{u_t}^2},\qquad
 x_{t+1}=x_t-\frac{u_t}{h_{t+1}}.
 \label{eq:D-update}
\end{equation}
There is no separate gradient initialization or additional query.
After $N\ge1$ completed local queries, the report $R_{D,N}$ is a uniformly
sampled member of $x_0,\ldots,x_{N-1}$, maintained by reservoir sampling
with coins independent of the gradient stream.
The report does not affect the updates, and both $x_t$ and $h_t$ persist
through every reporting boundary.

\begin{proposition}[Noiseless report bound]
\label{prop:noiseless-rate}
Suppose Assumptions~\ref{ass:objective} and~\ref{ass:variance} hold.
Every D update has length at most one, and, for each deterministic $N\ge1$,
\[
 \norm{\nabla f(R_{D,N})}
 \le \sqrt{2L\Deltaf}+L(N-1)
 \qquad\text{almost surely}.
\]
In particular, each finite-query report has finite expected gradient norm
in every oracle regime under consideration.
If the gradients are noiseless, define
\begin{equation}
 D_0=2+2\Deltaf+2L\log\bigl(\max\{1,2L\}\bigr),
 \qquad D_0\le4(1+\Deltaf+L^2).
 \label{eq:noiseless-coefficient}
\end{equation}
Then, at every deterministic integer local query count $N\ge1$,
\begin{equation}
 \E\norm{\nabla f(R_{D,N})}\le\frac{D_0}{\sqrt N}.
 \label{eq:noiseless-rate}
\end{equation}
The coefficient is used only to state the guarantee; it is not supplied
to the path.
\end{proposition}

\begin{proof}
Since $h_{t+1}^2=h_t^2+\norm{u_t}^2$ and $h_t\ge1$,
\[
 \norm{x_{t+1}-x_t}=\frac{\norm{u_t}}{h_{t+1}}\le1.
\]
Consequently $\norm{x_t-\xinit}\le t$.
Smoothness and the initial-gap assumption give
\[
 \norm{\nabla f(\xinit)}^2\le2L\Deltaf,
 \qquad
 \norm{\nabla f(x_t)}\le\sqrt{2L\Deltaf}+Lt.
\]
{Assumption~\ref{ass:objective} gives $L>0$.
Applying the descent inequality at $\xinit-\nabla f(\xinit)/L$ and using
the lower bound $f_*$ gives the first inequality.}
The displayed bound for a finite-query report follows because its support
is $x_0,\ldots,x_{N-1}$.

Now suppose the gradients are noiseless.
Conditional zero variance implies $u_t=\nabla f(x_t)$ almost surely at
every adaptive query; the countable intersection of these events still
has probability one.
On this event, write $s_t=\norm{\nabla f(x_t)}^2$.
The descent inequality applied to \eqref{eq:D-update} gives
\[
 f(x_{t+1})\le f(x_t)-\frac{s_t}{h_{t+1}}
                    +\frac L2\frac{s_t}{h_{t+1}^2}.
\]
All quantities are finite at each finite query count by the displacement
and gradient bounds just proved.
Telescoping and $f(x_N)\ge f_*$ therefore imply
\begin{equation}
 \sum_{t=0}^{N-1}\frac{s_t}{h_{t+1}}
 \le\Deltaf+\frac L2\sum_{t=0}^{N-1}
                              \frac{s_t}{h_{t+1}^2}.
 \label{eq:D-descent}
\end{equation}
The accumulator identity gives two deterministic estimates.
First,
\[
 \frac{s_t}{h_{t+1}}
 =\frac{(h_{t+1}-h_t)(h_{t+1}+h_t)}{h_{t+1}}
 \ge h_{t+1}-h_t.
\]
Second, because $1/z\ge1/h_{t+1}^2$ for
$z\in[h_t^2,h_{t+1}^2]$,
\[
 \frac{s_t}{h_{t+1}^2}
 =\frac{h_{t+1}^2-h_t^2}{h_{t+1}^2}
 \le\int_{h_t^2}^{h_{t+1}^2}\frac{dz}{z}
 =2\log\frac{h_{t+1}}{h_t}.
\]
Substituting both estimates into \eqref{eq:D-descent} and using $h_0=1$
yields
\begin{equation}
 h_N\le1+\Deltaf+L\log h_N.
 \label{eq:D-scale-implicit}
\end{equation}
Put $r=\max\{1,2L\}$.
Concavity of the logarithm implies
$\log h_N\le\log r+(h_N-r)/r\le\log r+h_N/r$.
Since $L/r\le1/2$, \eqref{eq:D-scale-implicit} gives
\[
 h_N\le1+\Deltaf+L\log r+\tfrac12h_N,
 \qquad h_N\le D_0.
\]
The sum of squared gradient norms thus satisfies
\[
 \sum_{t=0}^{N-1}\norm{\nabla f(x_t)}^2=h_N^2-1\le D_0^2.
\]
Conditional on the path, independent reservoir sampling and
Cauchy--Schwarz now give
\[
 \E\bigl[\norm{\nabla f(R_{D,N})}\mid x_0,\ldots,x_{N-1}\bigr]
 =\frac1N\sum_{t=0}^{N-1}\norm{\nabla f(x_t)}
 \le\left(\frac1N\sum_{t=0}^{N-1}\norm{\nabla f(x_t)}^2\right)^{1/2}
 \le\frac{D_0}{\sqrt N}.
\]
Taking expectation proves \eqref{eq:noiseless-rate}.
Finally, $\log(\max\{1,2L\})\le2L$ for $L>0$, so
\[
 D_0\le2+2\Deltaf+4L^2\le4(1+\Deltaf+L^2).
\]
This proves the coefficient bound as well.
\end{proof}

{
\hypersetup{linkcolor=MidnightBlue,citecolor=MidnightBlue,urlcolor=MidnightBlue}
\section{Linear validation and the validation workload}
\label{sec:linear-validation}

This appendix proves Theorem~\ref{thm:anytime} by combining the S, M, and D report bounds with Lemma~\ref{lem:selection} and the deterministic query clock. It also compares the cumulative validation costs of the linear selector and the sublinear selector in Section~\ref{sec:priority-extension}.

\subsection{The main schedule and its proof}

At the end of round $n=1,2,4,\ldots$, freeze all three reports
$X_S,X_M,X_D$. During each of rounds $n+1,\ldots,2n$, collect one
fresh validation gradient at each frozen point, alongside one individual
optimization query on each path. Form the three means
\[
 \bar g_i=\frac1n\sum_{k=1}^n g(X_i;\xi_{i,k}^{\rm val}),
 \qquad i\in\{S,M,D\},
\]
and publish the candidate with smallest $\norm{\bar g_i}$ at the end
of round $2n$, breaking ties in the order S, M, D.
The next reports are frozen after this publication.
Until the first publication the output is $\xinit$; at every later
query it is the last published point.
All validation samples are fresh and independent of the optimization
streams, and validation does not affect path evolution.
We call this the \emph{linear-validation variant}.
It requires no problem parameters or total budget.
The optimization points continue, and each optimization update still
uses at most two queries.

\begin{proposition}[Linear validation and explicit coefficients]
\label{prop:linear-validation}
Under the basic assumptions, at every checkpoint $n$ the published point
$Z_n^{\rm lin}$ satisfies
\begin{equation}
 \E\norm{\nabla f(Z_n^{\rm lin})}
 \le\min_{i\in\{S,M,D\}}
       \E\norm{\nabla f(X_i)}+\frac{\sqrt6\,\sigma}{\sqrt n}.
 \label{eq:linear-checkpoint}
\end{equation}
For every deterministic integer query budget $T\ge1$, define
\begin{align}
 B_S(T)&=24^{1/4}S T^{-1/4},\notag\\
 B_M(T)&=24^{1/3}M T^{-1/3},\notag\\
 B_D(T)&=\sqrt{24}D_0T^{-1/2},
 \label{eq:linear-path-bounds}
\end{align}
where $S,M,g_0$ are the explicit quantities in
Appendix~\ref{sec:explicit-constants}, and $D_0$ is given by \eqref{eq:noiseless-coefficient}.
Let $\mathcal I$ include S always, M when average smoothness holds,
and D when gradients are exact.
The same returned point $\widehat x_T^{\rm lin}$ then satisfies
\begin{equation}
 \E\norm{\nabla f(\widehat x_T^{\rm lin})}
 \le\min_{i\in\mathcal I}B_i(T)
      +\frac{12\sigma+3g_0}{\sqrt T}.
 \label{eq:linear-global}
\end{equation}
The set $\mathcal I$ is used only to state the bound; the algorithm
does not receive it.
\end{proposition}

\begin{proof}
Condition on the history when the three candidates are frozen.
Each remains fixed during validation, even though optimization proceeds.
Independence and centering within each fresh validation batch give
\[
 \E\bigl[\norm{\bar g_i-\nabla f(X_i)}^2
       \mid\text{freezing history}\bigr]\le\sigma^2/n.
\]
Lemma~\ref{lem:selection} proves
\eqref{eq:linear-checkpoint}.
In particular, that lemma does not require independence among the three
candidates: it first compares to their smallest true gradient norm and
then uses $\E\min_i\norm{\nabla f(X_i)}
\le\min_i\E\norm{\nabla f(X_i)}$.
The finite-gradient-risk requirement holds because all paths make
bounded displacements at finite query counts.

The first round uses the S/M initialization queries and one D update query, with no validation.
Every subsequent round uses exactly three optimization and three
validation queries: one dyadic checkpoint's collection interval ends
exactly where the next begins.
Thus the deterministic query count at the end of round $r\ge1$ is
\begin{equation}
 D_r^{\rm lin}=6r-3.
 \label{eq:linear-clock}
\end{equation}
The first publication occurs at $D_2^{\rm lin}=9$.
For $T\ge9$, let $n$ be the checkpoint of the most recent publication.
The next publication is at the end of round $4n$, so
\[
 12n-3=D_{2n}^{\rm lin}\le T<D_{4n}^{\rm lin}=24n-3,
 \qquad T/24<n\le T.
\]
This index is deterministic for deterministic $T$, including when
execution pauses inside a round or between the two queries of a pair.
Every path has exactly $n$ local queries at the freezing checkpoint.
Its report therefore has the same distribution and guarantees as in
the original schedule.
Substitute the path bounds $Sn^{-1/4}$, $Mn^{-1/3}$, and
$D_0n^{-1/2}$, whenever their respective assumptions hold, into
\eqref{eq:linear-checkpoint}.
Since $n^{-1}\le24/T$, these substitutions give
$B_i(T)$ and validation error at most
$\sqrt{6\cdot24}\,\sigma/\sqrt T=12\sigma/\sqrt T$.
Taking the minimum of the applicable bounds proves
\eqref{eq:linear-global} for $T\ge9$, with the nonnegative initial-budget
term added.
For $1\le T<9$, the output is $\xinit$ and its gradient norm is at most
$g_0=\sqrt{2L\Deltaf}$; the term $3g_0/\sqrt T$ alone covers it.
This proves the bound for all stated budgets.
\end{proof}

The selection cost in \eqref{eq:linear-global} is $12\sigma/\sqrt T$ in addition to the same path coefficients $S$, $M$, and $D_0$ used by sublinear validation. The S/M coefficients account for scale control and candidate reuse.

\begin{corollary}[Polynomial bounds for linear validation]
\label{cor:linear-parameters}
For $P,P_{\rm av},P_0$ defined in
\eqref{eq:explicit-parameter-envelope}, a universal numerical constant
$C$ gives
\begin{equation}
 \E\norm{\nabla f(\widehat x_T^{\rm lin})}\le C
 \begin{cases}
 P^{5/2}T^{-1/4},&\text{under the basic assumptions},\\
 P_{\rm av}^{5/2}T^{-1/3},&\text{under average smoothness},\\
 P_0T^{-1/2},&\text{for noiseless gradients}.
 \end{cases}
 \label{eq:linear-polynomial}
\end{equation}
For quantitative use, the computable coefficients in \eqref{eq:linear-global} are the relevant guarantee. One may take $C=10^9$ in this deliberately loose common envelope; it is not an estimate of actual performance.
\end{corollary}

\begin{proof}
The bounds in \eqref{eq:explicit-path-envelope} imply
$S\le2\cdot10^8P^{5/2}$ and
$M\le2\cdot10^8P_{\rm av}^{5/2}$.
Also $\sigma\le\sqrt P$ and $g_0\le\sqrt2 P^{3/4}$.
Since $P_{\rm av}\ge P\ge1$, $24^{1/4}<3$, and $24^{1/3}<3$,
the S and M alternatives in \eqref{eq:linear-global} give the first
two lines with constants less than $10^9$.

For exact gradients, $\sigma=0$ and $D_0\le4P_0$.
Thus
\[
 B_D(T)+3g_0/\sqrt T
 \le(4\sqrt{24}P_0+3\sqrt2 P_0^{3/4})T^{-1/2}
 \le25P_0T^{-1/2},
\]
which is covered by the same common constant.
\end{proof}

\subsection{Validation workload and parameter dependence}
\label{sec:validation-tradeoff}

The linear variant uses exactly $3n$ validation queries for checkpoint
$n$. At the end of round $r$, its cumulative validation count is
$3(r-1)$ out of $6r-3$ total queries.
At an arbitrary interruption this differs from half the total query
count by a bounded number of queries.
Thus a nonvanishing fraction of its oracle work is validation.

For sublinear validation, the complete workload at checkpoint $n$ is
$3\lceil n^{2/3}\rceil$. Appendix~\ref{sec:anytime-selector} proves
\begin{equation}
 V_{\rm sub}(T)\le\frac{6(T/3)^{2/3}}{1-2^{-2/3}}=O(T^{2/3}),
 \label{eq:priority-total-validation}
\end{equation}
including all queries in unfinished batches.

\begin{table}[ht]
\centering
\caption{Validation costs on the same S/M/D paths. $V(T)$ counts every executed validation query, including queries in unfinished batches.}
\label{tab:validation-work}
\begin{tabularx}{\textwidth}{@{}l>{\raggedright\arraybackslash}X>{\raggedright\arraybackslash}X@{}}
\toprule
 & Linear validation & Sublinear validation\\
\midrule
Work at checkpoint $n$ & $3n$ & $3\lceil n^{2/3}\rceil$\\
Cumulative $V(T)$ & $T/2+O(1)$ & $O(T^{2/3})$\\
Selection rule & Smallest norm of three means & Smallest norm of three means\\
Selection error & $\sqrt6\,\sigma n^{-1/2}$ & $\sqrt6\,\sigma/\sqrt{\lceil n^{2/3}\rceil}$\\
First publication & Query $9$ & Query $9$\\
\bottomrule
\end{tabularx}
\end{table}

Both selectors preserve endpoints and retain the three oracle-regime guarantees with the same polynomial parameter envelope.
Linear validation gives a smaller additive selection error, while sublinear validation makes the fraction of validation queries tend to zero.
}

\section{Proof for sublinear validation}
\label{sec:anytime-selector}

We prove Corollary~\ref{thm:priority-anytime} using the same path bounds as for linear validation.
Let $S,M,D_0,g_0$ be the coefficients in Appendix~\ref{sec:explicit-constants}, and let $B_i(T)$ and $\mathcal I$ be as in Proposition~\ref{prop:linear-validation}.
We first establish the more explicit bound
\begin{equation}
 \E\norm{\nabla f(\widehat x_T^{\rm sub})}
 \le\min_{i\in\mathcal I}B_i(T)
      +\sqrt6\,24^{1/3}\sigma T^{-1/3}
      +3g_0T^{-1/2},\qquad T\ge1.
 \label{eq:sublinear-global}
\end{equation}
The set $\mathcal I$ describes the applicable guarantees and is not an input.

\paragraph{Selection at a checkpoint.}
Condition on the history at the end of round $n$ when the reports are frozen.
They remain fixed during validation.
Each batch consists of $m_n$ fresh, centered errors with conditional second moments at most $\sigma^2$; hence
\[
 \E\bigl[\norm{\bar g_i-\nabla f(X_i)}^2
       \mid\text{freezing history}\bigr]\le\sigma^2/m_n.
\]
The path updates never use validation samples, so concurrent optimization does not change this bound.
At each finite checkpoint all three candidates have finite expected gradient norm: normalized S/M steps and D steps have deterministically bounded lengths, and objective smoothness bounds the gradients on the resulting finite-radius ball.
Lemma~\ref{lem:selection} therefore proves \eqref{eq:sublinear-checkpoint}, whether or not the candidates are independent.
Propositions~\ref{prop:block-rates} and~\ref{prop:noiseless-rate} bound the respective candidate risks by $Sn^{-1/4}$, $Mn^{-1/3}$, and $D_0n^{-1/2}$ whenever their assumptions hold.
Substituting any applicable bound in \eqref{eq:sublinear-checkpoint} preserves its exponent: $n^{-1/3}\le n^{-1/4}$ in the basic regime, while $\sigma=0$ eliminates the validation term in the noiseless regime.

\paragraph{From checkpoints to every query budget.}
Since $n$ is a positive integer and $n^{2/3}\le n$, we have $m_n\le n$.
Thus the means are complete before publication at round $2n$, and checkpoint intervals never overlap.
The first round uses three optimization queries and no validation.
Every later round uses three optimization queries and either zero or three validation queries.
The exact deterministic count at the end of round $r$ is
\begin{equation}
 Q_r=3r+3\sum_{\substack{n<r\\n\text{ a power of }2}}
                   \min\{m_n,r-n\},
 \qquad 3r\le Q_r\le6r-3.
 \label{eq:anytime-clock}
\end{equation}
The sum counts the gradients collected at each frozen candidate, including unfinished means.
The upper bound also follows by assigning at most three validation queries to each of rounds $2,\ldots,r$.
In particular, $Q_1=3$ and $Q_2=9$, since $m_1=1$.

For a deterministic $T\ge9$, let $n$ denote the most recently published checkpoint.
The next publication is at the end of round $4n$; consequently
\[
 Q_{2n}\le T<Q_{4n}\le24n-3<24n,
 \qquad n>T/24.
\]
All query counts and publication times are deterministic.
This argument therefore also applies when $T$ ends inside a round, an M pair, or a validation batch, without conditioning on a random stopping event.
Each report frozen at round $n$ has exactly $n$ local path queries.
Substitution in \eqref{eq:sublinear-checkpoint} now gives
\[
 \E\norm{\nabla f(\widehat x_T^{\rm sub})}
 \le\min_{i\in\mathcal I}B_i(T)
      +\sqrt6\,24^{1/3}\sigma T^{-1/3}.
\]
For $1\le T<9$, the returned point is $\xinit$ and its gradient norm is at most $g_0=\sqrt{2L\Deltaf}$.
The term $3g_0T^{-1/2}$ covers these budgets.
Adding this nonnegative term for later budgets proves \eqref{eq:sublinear-global}.

The coefficient bounds in \eqref{eq:explicit-path-envelope},
$\sigma\le\sqrt P$, $g_0\le\sqrt2 P^{3/4}$, and $P_{\rm av}\ge P\ge1$ give the first two lines of \eqref{eq:explicit-anytime-envelope}.
For example, the basic bound uses $T^{-1/3}\le T^{-1/4}$ and the average-smooth bound keeps $T^{-1/3}$.
For exact gradients, $\sigma=0$ and $D_0\le4P_0$, so
\[
 \E\norm{\nabla f(\widehat x_T^{\rm sub})}
 \le(4\sqrt{24}P_0+3\sqrt2 P_0^{3/4})T^{-1/2}
 \le25P_0T^{-1/2}.
\]
As for linear validation, $C=10^9$ suffices for a common polynomial envelope.
The selection error is additive and does not multiply different path coefficients.

\paragraph{Cumulative validation work.}
Let $V_{\rm sub}(T)$ count all validation queries performed by query $T$.
Each started checkpoint has $n\le T/3$, because its freezing round has already used $3n$ optimization queries.
Charging every started checkpoint its full cost also covers partially completed batches.
For $n\ge1$, $\lceil n^{2/3}\rceil\le2n^{2/3}$, so the dyadic geometric sum gives
\begin{equation}
 V_{\rm sub}(T)
 \le3\sum_{\substack{n\le T/3\\n\text{ a power of }2}}\lceil n^{2/3}\rceil
 \le\frac{6(T/3)^{2/3}}{1-2^{-2/3}}
 =O(T^{2/3}).
 \label{eq:sublinear-total-validation}
\end{equation}
For $T<3$, no checkpoint has started and the bound holds with $V_{\rm sub}(T)=0$.
Thus $V_{\rm sub}(T)/T\to0$.
For comparison, linear validation uses $3(r-1)$ validation queries among $6r-3$ total queries at the end of round $r$.
A partial round changes this count by a bounded number, giving $V_{\rm lin}(T)=T/2+O(1)$.
This proves the workload claims in Table~\ref{tab:validation-work} and completes Corollary~\ref{thm:priority-anytime}.

{
\hypersetup{linkcolor=MidnightBlue,citecolor=MidnightBlue,urlcolor=MidnightBlue}
\section{Explicit parameter dependence of the anytime guarantee}
\label{sec:explicit-constants}

This appendix gives explicit path coefficients shared by linear and sublinear validation. The selection bounds in \eqref{eq:linear-global} and \eqref{eq:sublinear-global} add validation error to the best applicable path bound, without multiplying coefficients from different paths.

\subsection{Coefficients before simplification}

For $\Lambda\ge0$, define the following deterministic quantities:
\begin{align}
 E_\Lambda&=4\sqrt{\sigma^2+\Lambda^2},\qquad
 c_\Lambda=E_\Lambda+L/2,\notag\\
 K_\Lambda&=2+\Deltaf/2+c_\Lambda
       +\frac{9c_\Lambda}{2}
          \log\bigl(\max\{1,9c_\Lambda/2\}\bigr),\notag\\
 A_\Lambda&=K_\Lambda^2
       +\frac{E_\Lambda^2}{4}(9\log K_\Lambda+2),\notag\\
 V_\Lambda&=2L\bigl(\Deltaf+c_\Lambda(9\log K_\Lambda+2)\bigr),\notag\\
 C_\gamma(\Lambda)&=12\sqrt{A_\Lambda V_\Lambda}
       +16\sqrt2\,9^\gamma(4+E_\Lambda)A_\Lambda,
       \qquad \gamma\in\{1/4,1/3\}.
 \label{eq:explicit-path-coefficients}
\end{align}
Equations~\eqref{eq:block-scale-moments} and
\eqref{eq:block-gradient-moment} give
$\E a_\infty^2\le A_\Lambda$ and $\E G_\infty^2\le V_\Lambda$.
Substituting these in \eqref{eq:coarse-reports} and \eqref{eq:fine-final},
and using $N^{-1}\le N^{-\gamma}$ for $N\ge1$, proves
\[
 \E\norm{\nabla f(R_{S,N})}\le S N^{-1/4},\qquad
 \E\norm{\nabla f(R_{M,N})}\le M N^{-1/3}.
\]
Here $S:=C_{1/4}(L)$ and $M:=C_{1/3}(\bar L)$; the M inequality requires average smoothness. No unspecified
problem-dependent constant remains in these coefficients.

The remaining coefficients are
\[
 g_0=\sqrt{2L\Deltaf},\qquad
 D_0=2+2\Deltaf+2L\log(\max\{1,2L\}),
\]
where $g_0$ bounds the gradient at the initial point and $D_0$ is the noiseless D-path coefficient.
Equations~\eqref{eq:linear-global} and~\eqref{eq:sublinear-global}, together with these explicit values of $S,M,D_0,g_0$, give computable guarantees for every deterministic budget, including budgets before the first publication.
Neither these quantities nor any upper bounds on them are inputs to either algorithm.

\subsection{A polynomial envelope}

Write
\begin{equation}
 P=1+\Deltaf+L^2+\sigma^2,\qquad
 P_{\rm av}=P+\bar L^2,\qquad P_0=1+\Deltaf+L^2.
 \label{eq:explicit-parameter-envelope}
\end{equation}
The path coefficients satisfy
\begin{equation}
 S\le2\cdot10^8P^{5/2},\qquad
 M\le2\cdot10^8P_{\rm av}^{5/2},\qquad
 D_0\le4P_0.
 \label{eq:parameter-path-envelope}
\end{equation}
These bounds imply the same polynomial envelopes for both validation rules, as shown in Appendices~\ref{sec:linear-validation} and~\ref{sec:anytime-selector}.
For quantitative use, the explicit coefficients above are more informative than the polynomial simplification.

\begin{proof}
Let $W\ge1$ satisfy
$\Deltaf\le W$, $L\le\sqrt W$, and $\sigma^2+\Lambda^2\le W$.
Then $E_\Lambda\le4\sqrt W$ and
$c_\Lambda\le(9/2)\sqrt W$.
For $r=9c_\Lambda/2$, the inequality
$r\log(\max\{1,r\})\le r^2$ holds for all $r\ge0$.
Consequently,
\[
 K_\Lambda\le
 2+\tfrac12W+\tfrac92\sqrt W+\tfrac{6561}{16}W
 \le418W.
\]
Using $\log K_\Lambda\le K_\Lambda$ in
\eqref{eq:explicit-path-coefficients} gives
\begin{align*}
 A_\Lambda
 &\le418^2W^2+4W(9\cdot418W+2)
 \le190000W^2,\\
 V_\Lambda
 &\le2\sqrt W\bigl[W+\tfrac92\sqrt W(9\cdot418W+2)\bigr]
 \le33878W^2.
\end{align*}
Since $9^\gamma\le3$, $16\sqrt2<23$, and
$4+E_\Lambda\le8\sqrt W$, substitution proves
\begin{equation}
 C_\gamma(\Lambda)\le2\cdot10^8 W^{5/2},
 \qquad \gamma\in\{1/4,1/3\}.
 \label{eq:explicit-path-envelope}
\end{equation}
For S take $W=P$, and for M take $W=P_{\rm av}$.
The bound on $D_0$ is proved in Proposition~\ref{prop:noiseless-rate}.
This establishes \eqref{eq:parameter-path-envelope}.
\end{proof}

These powers are sufficient bounds from the present analysis, not lower bounds on adaptation or continuation.
The uniform scale-moment control and the replacement of logarithmic parameter factors by polynomials contribute to their size.
Both selectors retain the same path coefficients; sublinear validation has a larger additive selection error.
}

{
\hypersetup{linkcolor=MidnightBlue,citecolor=MidnightBlue,urlcolor=MidnightBlue}
\section{The restarting anytime baseline}
\label{sec:restarting-baseline}

The fixed-budget guarantee already gives an anytime method if resetting the
optimization point is allowed.  The following construction makes its query
cost and its comparison with Theorem~\ref{thm:anytime} explicit.

\begin{proposition}[Doubling with restarts]
\label{prop:restarting-baseline}
Run Algorithm~\ref{alg:three-stage} successively with nominal budgets
$7\cdot2^j$, $j=0,1,2,\ldots$, using fresh randomness and starting each run
at $\xinit$.  Return the last completed run's output, or $\xinit$ before
the first run finishes.  At every deterministic integer query budget
$T\ge1$, this rule returns a point $\widehat x_T^{\mathrm{restart}}$ with
\begin{equation}
 \E\norm{\nabla f(\widehat x_T^{\mathrm{restart}})}\le
 \begin{cases}
 (\Deltaf+3L+7\sigma)(32/T)^{1/4},&\text{under the basic assumptions},\\
 (\Deltaf+3L+3\bar L+12\sigma)(32/T)^{1/3},&\text{under average smoothness},\\
 (\Deltaf+4L)(32/T)^{1/2},&\text{for noiseless gradients}.
 \end{cases}
 \label{eq:restarting-rates}
\end{equation}
All initialization and validation queries are counted.  Each optimization
update uses at most two gradient evaluations, and growing batches occur
only in candidate validation.  The rule needs no total budget or problem
parameters.
\end{proposition}

\begin{proof}
Run $j$ uses $n_j=2^j$ updates in each stage and exactly
$c_j=7\cdot2^j-1$ queries, including validation, by
Theorem~\ref{thm:simultaneous}.  Its deterministic completion time is
\[
 C_j=\sum_{i=0}^j c_i=14\cdot2^j-8-j.
\]
For $T\ge C_0=6$, let $j$ be the unique index with
$C_j\le T<C_{j+1}$.  Then the returned point is the output of run $j$,
and
\[
 T<C_{j+1}=28\cdot2^j-9-j<32n_j.
\]
Apply each applicable line of Theorem~\ref{thm:simultaneous} to this
deterministically selected run and use
$n_j^{-p}\le(32/T)^p$ for $p\in\{1/4,1/3,1/2\}$.
No union bound over runs is needed.  The argument applies while the next
run is still optimizing or validating, including interruptions between the
two evaluations of a shared-sample pair: those queries only increase $T$,
and the previous output remains available.

For $1\le T<6$, smoothness and the lower bound on $f$ give
$\norm{\nabla f(\xinit)}\le\sqrt{2L\Deltaf}\le\Deltaf+L$.
Every coefficient in \eqref{eq:restarting-rates} is at least
$\Deltaf+L$, and $(32/T)^p\ge1$, proving these cases as well.
The per-update and validation costs follow directly from
Algorithm~\ref{alg:three-stage}.
\end{proof}
}

\section{Two continuing paths when logarithmic factors are allowed}
\label{sec:adagrad-comparison}
The three-path construction does not assert that three paths are necessary
for the three rate exponents. This appendix makes precise a two-path
comparison using the existing M and D workers, with an additional
$\sqrt{\log(eT)}$ factor in the basic guarantee.

\paragraph{A noisy-oracle guarantee for D.}
The update~\eqref{eq:D-path-update} is exactly AdaGrad-Norm with unit
base step size and initial squared accumulator one.
Under Assumptions~\ref{ass:objective} and~\ref{ass:variance}, its conditional
second moment satisfies
$\E[\norm{u_t}^2\mid\mathcal F_t]\le\sigma^2+\norm{\nabla f(x_t)}^2$,
where $\mathcal F_t$ is the history before drawing $u_t$.
Thus the affine second-moment assumption in
\cite{wang2023adagrad} holds with their $D_0=\sigma^2$ and $D_1=1$.
The proof of their Theorem~2, Stage~II, gives finite constants
$c_1,c_3\ge1$ and $c_2\ge0$, depending only on the fixed problem quantities
and initialization, such that
\begin{equation}
 \left(\E\sqrt{\sum_{t=0}^{N-1}\norm{\nabla f(x_t)}^2}\right)^2
 \le H_N\bigl(c_1+2c_2\log H_N\bigr),
 \qquad H_N=2\sqrt{2\sigma^2N+1}+c_3.
 \label{eq:adagrad-moment-comparison}
\end{equation}
Specifically, these can be their constants $C_1,C_2,C_3$ specialized
to base step size and initial squared accumulator one, with $C_1,C_3$
enlarged to at least one.
This is the moment estimate immediately preceding the Markov step in
that proof; the stated high-probability minimum-iterate conclusion alone
would not suffice for our reporting rule.
Independent uniform reservoir sampling and Cauchy--Schwarz give
\begin{equation}
 \begin{aligned}
 \E\norm{\nabla f(R_{D,N})}
 &=\frac1N\E\sum_{t=0}^{N-1}\norm{\nabla f(x_t)}\\
 &\le\left(\frac{H_N(c_1+2c_2\log H_N)}{N}\right)^{1/2}
 =O\!\left(\frac{\sqrt{\log(eN)}}{N^{1/4}}\right).
 \end{aligned}
 \label{eq:adagrad-report-comparison}
\end{equation}
All implied constants are independent of $N$; no bounded-gradient
assumption or known problem parameter is needed. For exact gradients,
Proposition~\ref{prop:noiseless-rate} supplies the sharper coefficient
$D_0N^{-1/2}$ already used in the paper.

\paragraph{Two paths and one returned point.}
Run only M and D, with one individual query per path in each round.
At every dyadic round $n$, freeze their reports. During the next
$m_n=\lceil n^{2/3}\rceil$ rounds, collect one fresh gradient at each
frozen report per round; publish the candidate whose empirical mean has
smaller norm at the end of round $2n$. Since $m_n\le n$, both means are ready
then. Retain the last published output at interruptions, returning
$\xinit$ before the first publication. As in Algorithm~\ref{alg:anytime},
complete enabled updates and publication before honoring an interruption.
Neither validation nor publication changes an optimization endpoint.

For fixed frozen candidates $X_M,X_D$, let
$e_i=\norm{\bar g_i-\nabla f(X_i)}$ and let $j$ be the selected index.
For either competitor $i$, if $j\ne i$, the triangle inequality gives
$\norm{\nabla f(X_j)}\le\norm{\nabla f(X_i)}+e_i+e_j
=\norm{\nabla f(X_i)}+e_M+e_D$; if $j=i$, the bound with $e_M+e_D$
holds directly. Conditional on the frozen reports, each mean has
expected error at most $\sigma/\sqrt{m_n}$. Hence, writing $\widehat X_n$
for the published candidate,
\begin{equation}
 \E\norm{\nabla f(\widehat X_n)}
 \le\min_{i\in\{M,D\}}\E\norm{\nabla f(R_{i,n})}
       +\frac{2\sigma}{\sqrt{m_n}}.
 \label{eq:two-path-selection}
\end{equation}
This also holds when only one of the two report bounds is used.

At the end of round $r$, the deterministic total query count is
\[
 Q_r=2r+2\sum_{\substack{n<r\\n\text{ dyadic}}}
                \min\{m_n,r-n\},\qquad 2r\le Q_r\le4r-2.
\]
The validation intervals $(n,2n]$ are disjoint, proving the upper bound.
The first publication is at $Q_2=6$. For any deterministic total budget
$T\ge6$, let $n$ be the label of the latest publication. The next
publication is at round $4n$, so $T<Q_{4n}\le16n-2$ and $n>T/16$.
Combining~\eqref{eq:two-path-selection} with
\eqref{eq:adagrad-report-comparison}, Proposition~\ref{prop:block-rates},
and Proposition~\ref{prop:noiseless-rate} proves, for the same returned
point, the rates
\[
 O\!\left(\sqrt{\log(eT)}\,T^{-1/4}\right),\qquad
 O(T^{-1/3}),\qquad O(T^{-1/2})
\]
under the basic assumptions, average smoothness, and exact gradients,
respectively. The finitely many budgets $T<6$ are covered by
$\norm{\nabla f(\xinit)}\le\sqrt{2L\Deltaf}$.
Charging every started batch its full $2m_n$ queries gives total
validation work $O(T^{2/3})$, including unfinished batches.
Thus two continuing paths suffice if a horizon logarithm is allowed;
path S provides our log-free basic guarantee.

\section*{AI methodology statement}
\label{sec:ai-methodology}
In this work, we used generative AI tools (OpenAI Codex) to check mathematical arguments, assist with corrections, and assist with the presentation of existing proofs. Additionally, we used generative AI tools to improve writing clarity and find relevant literature. We take responsibility for the final content of this work and have reviewed all AI-assisted content.

\clearpage
\bibliographystyle{alpha}
\bibliography{references}
\label{page:references-end}
\end{document}